\documentclass[12pt, leqno, twoside]{article}
\usepackage{amsmath, amsthm, amsfonts, amssymb, graphicx, color}
\numberwithin{equation}{section}
     \newtheorem{thm}{Theorem}[section]
     \newtheorem{cor}[thm]{Corollary}
     \newtheorem{prop}[thm]{Proposition}
     \newtheorem{lem}[thm]{Lemma}
\theoremstyle{definition}
      \newtheorem{defn}{Definition}[section]
     \newtheorem{exmp}{Example}[section]
\theoremstyle{remark}
     \newtheorem{rem}{Remark}[section]

\newcommand{\R}{\mathbb{R}}

\newcommand{\Z}{\mathbb{Z}}

\newcommand{\cL}{\mathcal{L}}

\newcommand{\cY}{\mathcal{Y}}

\newcommand{\BMO}{\mathrm{BMO}}
\newcommand{\Lip}{\mathrm{Lip}}

\newcommand{\loc}{\mathrm{loc}}
\newcommand{\comp}{\mathrm{comp}}

\newcommand{\ls}{\lesssim}
\newcommand{\gs}{\gtrsim}

\newcommand{\tp}{\tilde{p}}
\newcommand{\tq}{\tilde{q}}

\newcommand{\Cic}{C^{\infty}_{\comp}}
\newcommand{\Lic}{L^{\infty}_{\comp}}
\newcommand{\Li}{L^{\infty}}
\newcommand{\LP}{L^{\Phi}}
\newcommand{\LPs}{L^{\Psi}}
\newcommand{\LT}{L^{\Theta}}
\newcommand{\LcP}{L^{\widetilde\Phi}}
\newcommand{\LcPs}{L^{\widetilde\Psi}}
\newcommand{\cPhi}{\widetilde{\Phi}}
\newcommand{\cPsi}{\widetilde{\Psi}}
   
\newcommand{\wL}{\mathrm{w}\hskip-0.6pt{L}}
\newcommand{\wLP}{\mathrm{w}\hskip-0.6pt{L}^{\Phi}}
\newcommand{\wLPs}{\mathrm{w}\hskip-0.6pt{L}^{\Psi}}

\newcommand{\Ir}{I_{\rho}}
\newcommand{\Ia}{I_{\alpha}}

\newcommand{\Ma}{M_{\alpha}}
\newcommand{\Mr}{M_{\rho}}
\newcommand{\Ms}{M^{\sharp}}
\newcommand{\Md}{M^{\mathrm{dy}}}

\newcommand{\cQd}{\mathcal{Q}^{\mathrm{dy}}}

\newcommand{\mint}{-\hspace{-13pt}\int}

\newcommand{\dlim}{\displaystyle\lim}

\newcommand{\sgn}{\mathrm{sgn}}

\newcommand{\biPhi}{{\it{\bar\Phi}}}
\newcommand{\biP}{{\it{\bar\Phi}}}
\newcommand{\iPy}{{\it{\Phi_Y}}}

\newcommand{\biPy}{{\it{\bar\Phi_Y}}}

\newcommand{\dtwo}{\Delta_2}
\newcommand{\ntwo}{\nabla_2}
\newcommand{\bdtwo}{\bar\Delta_2}
\newcommand{\bntwo}{\bar\nabla_2}

\newcommand{\msckw}{%
\footnotetext{\hspace{-0.35cm} 2010 {\it Mathematics Subject Classification}. 
46E30, 42B35.
\endgraf
{\it Key words and phrases.} 
Orlicz space, Campanato space, fractional integral, commutator. 
\endgraf
Minglei Shi,
18nd206l@vc.ibaraki.ac.jp, stfoursml@gmail.com
\endgraf
Ryutaro Arai,
18nd201t@vc.ibaraki.ac.jp, araryu314159@gmail.com
\endgraf
Eiichi Nakai,
eiichi.nakai.math@vc.ibaraki.ac.jp
}
}

\begin{document}

\title{Generalized fractional integral operators 
and their commutators with functions in generalized Campanato spaces on Orlicz spaces \msckw}
\author{Minglei Shi, Ryutaro Arai and Eiichi Nakai\footnote{Corresponding author} \\
{\small Department of Mathematics, Ibaraki University, Mito, Ibaraki 310-8512, Japan}}
\date{}

\maketitle

\begin{abstract}
We investigate the commutators $[b,I_{\rho}]$ 
of generalized fractional integral operators $I_{\rho}$ with functions $b$ in
generalized Campanato spaces
and give a necessary and sufficient condition for the boundedness of the commutators
on Orlicz spaces.
To do this we
define Orlicz spaces with generalized Young functions
and prove the boundedness of generalized fractional maximal operators on the Orlicz spaces.
\end{abstract}

\section{Introduction}\label{sec:intro}

Let $\R^n$ be the $n$-dimensional Euclidean space,
and let $\Ia$ be the fractional integral operator of order $\alpha\in(0,n)$, that is,
\begin{equation*}
 \Ia f(x)=\int_{\R^n} \frac{f(y)}{|x-y|^{n-\alpha}} \,dy,
 \quad x\in\R^n.
\end{equation*}
Then it is known as the Hardy-Littlewood-Sobolev theorem that
$\Ia$ is bounded from $L^p(\R^n)$ to $L^q(\R^n)$,
if $\alpha\in(0,n)$, $p,q\in(1,\infty)$ and $-n/p+\alpha=-n/q$.
This boundedness 
was extended to Orlicz spaces 
by several authors, see
\cite{Cianchi1999,Edmunds-Gurka-Opic1995,Kokilashvili-Krbec1991,ONeil1965,Strichartz1972,Torchinsky1976,Trudinger1967}, etc.
Chanillo~\cite{Chanillo1982} considerd the commutator
\begin{equation*}
 [b,\Ia]f=b\Ia f-\Ia(bf),
\end{equation*}
with $b\in\BMO$
and proved that $[b,\Ia]$ has the same boundedness as $\Ia$. 
The result was also extended to Orlicz spaces by 
Fu, Yang and Yuan~\cite{Fu-Yang-Yuan2014}
and Guliyev, Deringoz and Hasanov~\cite{Guliyev-Deringoz-Hasanov2017}.

In this paper we consider generalized fractional integral operators
$\Ir$ on Orlicz spaces.
For a function $\rho:(0,\infty)\to(0,\infty)$,
the operator $\Ir$ is defined by
\begin{equation}\label{Ir}
 \Ir f(x)=\int_{\R^n}\frac{\rho(|x-y|)}{|x-y|^n}f(y)\,dy,
 \quad x\in\R^n,
\end{equation}
where we always assume that
\begin{equation}\label{int rho}
 \int_0^1\frac{\rho(t)}{t}\,dt<\infty.
\end{equation}
If $\rho(r)=r^{\alpha}$, $0<\alpha<n$, 
then $\Ir$ is the usual fractional integral operator $\Ia$.
The condition \eqref{int rho} is needed for the integral in \eqref{Ir} 
to converge for bounded functions $f$ with compact support.
In this paper we also assume that 
there exist positive constants $C$, $K_1$ and $K_2$ with $K_1<K_2$ such that, for all $r>0$,
\begin{equation}\label{sup rho}
 \sup_{r\le t\le 2r}\rho(t)
 \le
 C\int_{K_1r}^{K_2r}\frac{\rho(t)}{t}\,dt.
\end{equation}

The operator $\Ir$ was introduced in \cite{Nakai2001Taiwan} 
to extend the Hardy-Littlewood-Sobolev theorem to Orlicz spaces
whose partial results were announced in 
\cite{Nakai2000ISAAC}.
For example, 
the generalized fractional integral $I_{\rho}$ 
is bounded from $\exp L^p(\R^n)$ to $\exp L^q(\R^n)$,
where 
\begin{equation}\label{rho log}
     \rho(r)=
     \begin{cases}
          1/(\log(1/r))^{\alpha+1} & \text{for small}\; r, \\
          (\log r)^{\alpha-1}     & \text{for large}\; r,
     \end{cases}
     \quad \alpha>0,
\end{equation}
$p,q\in(0,\infty)$, $-1/p+\alpha=-1/q$ 
and $\exp L^p(\R^n)$ is the Orlicz space $L^{\Phi}(\R^n)$ with 
\begin{equation}\label{Phi exp}
     \Phi(r)=
     \begin{cases}
          1/\exp(1/r^p) & \text{for small}\; r, \\
          \exp(r^p)     & \text{for large}\; r.
     \end{cases}
\end{equation}
See also \cite{Nakai2001SCMJ,Nakai2002Lund,Nakai2004KIT,Nakai2008Studia,Nakai-Sumitomo2001SCMJ}.
Recently, in \cite{Deringoz-Guliyev-Nakai-Sawano-Shi-preprint}  
some necessary and sufficient conditions for the boundedness of $\Ir$ on Orlicz spaces
have been given.

In this paper we consider the commutator $[b,\Ir]$ with a function $b$ in generalized Campanato spaces.
To prove the boundedness of $[b,\Ir]$ on Orlicz spaces
we need the sharp maximal operator $\Ms$ and
generalized fractional maximal operators $\Mr$,
see \eqref{Ms} and \eqref{Mr} below for their definitions.
Moreover, we need a generalization of the Young function.

First we recall the definition of the generalized Campanato space and
the sharp maximal and generalized fractional maximal operators.
We denote by $B(x,r)$ 
the open ball centered at $x\in\R^n$ and of radius $r$,
that is, 
\begin{equation*}
 B(x,r)
 =\{y\in\R^n:|y-x|<r\}.
\end{equation*}
For a measurable set $G \subset \R^n$, 
we denote by $|G|$ and $\chi_{G}$ 
the Lebesgue measure of $G$ and the characteristic function of $G$, 
respectively.
For a function $f\in L^1_{\loc}(\R^n)$ and a ball $B$, 
let 
\begin{equation*}\label{mean}
  f_{B}=\mint_{B} f=\mint_{B} f(y)\,dy=\frac1{|B|}\int_{B} f(y)\,dy.
\end{equation*}

\begin{defn}\label{defn:gC}
For $p\in[1,\infty)$ and $\psi:(0,\infty)\to(0,\infty)$, 
let
$\cL_{p,\psi}(\R^n)$ be the set of all functions $f$ such that the 
following functional is finite: 
\begin{equation*}
 \|f\|_{\cL_{p,\psi}(\R^n)}
  =\sup_{B=B(x,r)}\frac1{\psi(r)}\left(\mint_{B}|f(y)-f_{B}|^p\,dy\right)^{1/p},
\end{equation*}
where
the supremum is taken over all balls $B(x,r)$ in $\R^n$.
\end{defn}
Then $\|f\|_{\cL_{p,\psi}(\R^n)}$ is a norm modulo constant functions 
and thereby $\cL_{p,\psi}(\R^n)$ is a Banach space.
If $p=1$ and $\psi\equiv1$, then $\cL_{p,\psi}(\R^n)=\BMO(\R^n)$.
If $p=1$ and $\psi(r)=r^{\alpha}$ ($0<\alpha\le1$), 
then $\cL_{p,\psi}(\R^n)$ coincides with $\Lip_{\alpha}(\R^n)$.

The sharp maximal operator $\Ms$ is defined by 
\begin{equation}\label{Ms}
 \Ms f(x)=\sup_{B\ni x}\mint_{B}|f(y)-f_B|\,dy,
 \quad x\in\R^n,
\end{equation}
where the supremum is taken over all balls $B$ containing $x$.
For a function $\rho:(0,\infty)\to(0,\infty)$, let
\begin{equation}\label{Mr}
 \Mr f(x)=\sup_{B(z,r)\ni x}\rho(r)\mint_{B(z,r)}|f(y)|\,dy,
 \quad x\in\R^n,
\end{equation}
where the supremum is taken over all balls $B(z,r)$ containing $x$.
We don't assume the condition \eqref{int rho} or \eqref{sup rho} on the definition of $\Mr$.
The operator $\Mr$ was studied in \cite{Sawano-Sugano-Tanaka2011} on generalized Morrey spaces.
If $\rho(r)=|B(0,r)|^{\alpha/n}$, 
then $M_{\rho}$ is the usual fractional maximal operator $\Ma$.
If $\rho\equiv1$, then $\Mr$ is 
the Hardy-Littlewood maximal operator $M$, that is,
\begin{equation*}
 Mf(x)=\sup_{B\ni x}\mint_{B}|f(y)|\,dy,
 \quad x\in\R^n.
\end{equation*}

It is known that 
the usual fractional maximal operator $\Ma$ 
is dominated pointwise by 
the fractional integral operator $\Ia$,
that is,
$\Ma f(x)\le C\Ia|f|(x)$ for all $x\in\R^n$.
Then the boundedness of $\Ma$ follows from one of $\Ia$.
However, 
we need a better estimate on $\Mr$ than $\Ir$
to prove the boundedness of the commutator $[b,\Ir]$.
In this paper 
we give a necessary and sufficient condition of the boundedness of $\Mr$ 
which sharpens the result in \cite{Deringoz-Guliyev-Nakai-Sawano-Shi-preprint}.

The organization of this paper is as follows. 
In Section~\ref{sec:Young} 
we recall the definition of the Young function and give its generalization.
Then 
we define Orlicz spaces with generalized Young functions.
We state main results in Section~\ref{sec:main}. 
We give some lemmas in Section~\ref{sec:lemmas}
to prove the main results.
The boundedness of $\Ir$ has been proved in \cite{Deringoz-Guliyev-Nakai-Sawano-Shi-preprint}.
We prove the boundedness of $\Mr$ in Section~\ref{sec:proofs bdd}.
Moreover, 
we investigate pointwise estimate by using the sharp maximal operator 
and the norm estimate by the sharp maximal operator in Section~\ref{sec:sharp}.
Finally, using the generalized Young function 
and the results in Sections~\ref{sec:lemmas}--\ref{sec:sharp},
we prove the boundedness of $[b,\Ir]$ in Section~\ref{sec:proof comm}.

At the end of this section, we make some conventions. 
Throughout this paper, we always use $C$ to denote a positive constant 
that is independent of the main parameters involved 
but whose value may differ from line to line.
Constants with subscripts, such as $C_p$, is dependent on the subscripts.
If $f\le Cg$, we then write $f\ls g$ or $g\gs f$; 
and if $f \ls g\ls f$, we then write $f\sim g$.

\section{Generalization of the Young function and Orlicz spaces}\label{sec:Young}

First we define a set $\biP$ of increasing functions 
$\Phi:[0,\infty]\to[0,\infty]$
and give some properties of functions in $\biP$.

For an increasing function 
$\Phi:[0,\infty]\to[0,\infty]$,
let
\begin{equation*} 
 a(\Phi)=\sup\{t\ge0:\Phi(t)=0\}, \quad 
 b(\Phi)=\inf\{t\ge0:\Phi(t)=\infty\},
\end{equation*} 
with convention $\sup\emptyset=0$ and $\inf\emptyset=\infty$.
Then $0\le a(\Phi)\le b(\Phi)\le\infty$.
Let $\biP$ be the set of all increasing functions
$\Phi:[0,\infty]\to[0,\infty]$
such that
\begin{align}\label{ab}
 &0\le a(\Phi)<\infty, \quad 0<b(\Phi)\le\infty, \\
 &\lim_{t\to+0}\Phi(t)=\Phi(0)=0, \label{lim_0} \\
 &\text{$\Phi$ is left continuous on $[0,b(\Phi))$}, \label{left cont} \\
 &\text{if $b(\Phi)=\infty$, then } 
 \lim_{t\to\infty}\Phi(t)=\Phi(\infty)=\infty, \label{left cont infty} \\
 &\text{if $b(\Phi)<\infty$, then } 
 \lim_{t\to b(\Phi)-0}\Phi(t)=\Phi(b(\Phi)) \ (\le\infty). \label{left cont b}
\end{align}

In what follows,
if an increasing and left continuous function $\Phi:[0,\infty)\to[0,\infty)$ satisfies
\eqref{lim_0} and $\dlim_{t\to\infty}\Phi(t)=\infty$,
then we always regard that $\Phi(\infty)=\infty$ and that $\Phi\in\biP$.

For $\Phi\in\biP$,
we recall the generalized inverse of $\Phi$
in the sense of O'Neil \cite[Definition~1.2]{ONeil1965}.

\begin{defn}\label{defn:ginverse}
For $\Phi\in\biP$ and $u\in[0,\infty]$, let
\begin{equation}\label{inverse}
 \Phi^{-1}(u)
 = 
\begin{cases}
 \inf\{t\ge0: \Phi(t)>u\}, & u\in[0,\infty), \\
 \infty, & u=\infty.
\end{cases}
\end{equation}
\end{defn}

Let $\Phi\in\biP$. 
Then $\Phi^{-1}$ is finite, increasing and right continuous on $[0,\infty)$
and positive on $(0,\infty)$.
If $\Phi$ is bijective from $[0,\infty]$ to itself, 
then $\Phi^{-1}$ is the usual inverse function of $\Phi$.
Moreover, 
we have
the following proposition,
which is a generalization of Property 1.3 in \cite{ONeil1965}.

\begin{prop}\label{prop:inverse}
Let $\Phi\in\biP$.
Then
\begin{equation}\label{inverse ineq}
 \Phi(\Phi^{-1}(u)) \le u \le  \Phi^{-1}(\Phi(u))
 \quad\text{for all $u\in[0,\infty]$}.
\end{equation}
\end{prop}

\begin{proof}
First we show that, 
for all $t,u\in[0,\infty]$, 
\begin{equation}\label{inv 1}
 \Phi(t)\le u \ \Rightarrow \ t\le\Phi^{-1}(u).
\end{equation}
If $\Phi(t)\le u$, then $\Phi(s)>u\Rightarrow \Phi(s)>\Phi(t)\Rightarrow s>t$ and
\begin{equation*}
 \{s\ge0:\Phi(s)>u\}\subset\{s\ge0:s>t\}.
\end{equation*}
Hence,
\begin{equation*}
 \Phi^{-1}(u)=\inf\{s\ge0:\Phi(s)>u\}\ge\inf\{s\ge0:s>t\}=t.
\end{equation*}
This shows \eqref{inv 1}. 
Now, letting $\Phi(t)=u$ and using \eqref{inv 1}, we have that $t\le\Phi^{-1}(u)=\Phi^{-1}(\Phi(t))$,
which is the second inequality in \eqref{inverse ineq}.

Next 
we show that,
for all $t\in(0,\infty]$ and $u\in[0,\infty]$, 
\begin{align}\label{inv 2}
 \Phi(t)>u \ &\Rightarrow \ t>\Phi^{-1}(u), \\
 t\le \Phi^{-1}(u) \ &\Rightarrow \ \Phi(t)\le u.\label{inv 3}
\end{align}
We only show \eqref{inv 2}, since \eqref{inv 3} is equivalent to \eqref{inv 2}.
If $\Phi(t)>u$, then $\Phi(s)>u$ for some $s<t$
by the properties \eqref{left cont}--\eqref{left cont b}.
By the definition of $\Phi^{-1}$ we have that
$s\ge\Phi^{-1}(u)$. 
That is, $t>\Phi^{-1}(u)$, which shows \eqref{inv 2}.
Now, if $\Phi^{-1}(u)=0$, then the first inequality in \eqref{inverse ineq} is true by \eqref{lim_0}. 
If $t=\Phi^{-1}(u)>0$, then, using \eqref{inv 3}, we have that $\Phi(\Phi^{-1}(u))=\Phi(t)\le u$,
which is the first inequality in \eqref{inverse ineq}.
\end{proof}

For $\Phi, \Psi\in\biP$, 
we write $\Phi\approx\Psi$
if there exists a positive constant $C$ such that
\begin{equation*} 
     \Phi(C^{-1}t)\le\Psi(t)\le\Phi(Ct)
     \quad\text{for all}\ t\in[0,\infty].
\end{equation*} 
For functions $P,Q:[0,\infty]\to[0,\infty]$, 
we write $P\sim Q$ 
if there exists a positive constant $C$ such that
\begin{equation*} 
     C^{-1}P(t)\le Q(t)\le CP(t)
     \quad\text{for all}\ t\in[0,\infty].
\end{equation*} 
Then, for $\Phi,\Psi\in\biP$, 
\begin{equation}\label{approx equiv}
 \Phi\approx\Psi \quad \Leftrightarrow \quad \Phi^{-1}\sim\Psi^{-1}.
\end{equation}
Actually we have the following lemma.

\begin{lem}\label{lem:inverse}
Let $\Phi,\Psi\in\biP$, and let $C$ be a positive constant.
Then 
\begin{equation*}
 \Phi(t)\le\Psi(Ct) \quad\text{for all $t\in[0,\infty]$} 
\end{equation*}
if and only if
\begin{equation*}
 \Psi^{-1}(u)\le C\Phi^{-1}(u) \quad\text{for all $u\in[0,\infty]$}.
\end{equation*}
\end{lem}

\begin{proof}
Let $\Phi(t)\le\Psi(Ct)$ for all $t\in[0,\infty]$. 
If $t=\Psi^{-1}(u)$, then 
by Proposition~\ref{prop:inverse} we have 
that $\Psi(t)=\Psi(\Psi^{-1}(u))\le u$ and that
\begin{equation*}
 \Psi^{-1}(u)/C=t/C\le\Phi^{-1}(\Phi(t/C))\le\Phi^{-1}(\Psi(t))\le\Phi^{-1}(u).
\end{equation*}

Conversely, let $\Psi^{-1}(u)\le C\Phi^{-1}(u)$ for all $u\in[0,\infty]$.
If $u=\Psi(t)$, then 
by Proposition~\ref{prop:inverse} we have 
$t\le \Psi^{-1}(\Psi(t))=\Psi^{-1}(u)$ and
\begin{equation*}
 \Phi(t/C)\le\Phi(\Psi^{-1}(u)/C)\le\Phi(\Phi^{-1}(u))\le u=\Psi(t).
\qedhere
\end{equation*}
\end{proof}

Next we recall the definition of the Young function and give its generalization.

\begin{defn}\label{defn:Young}
A function $\Phi\in\biPhi$ is called a Young function 
(or sometimes also called an Orlicz function) 
if 
$\Phi$ is convex on $[0,b(\Phi))$.
\end{defn}

By the convexity, 
any Young function $\Phi$ is continuous on $[0,b(\Phi))$ and strictly increasing on $[a(\Phi),b(\Phi)]$.
Hence $\Phi$ is bijective from $[a(\Phi),b(\Phi)]$ to $[0,\Phi(b(\Phi))]$.
Moreover, $\Phi$ is absolutely continuous on any closed subinterval in $[0,b(\Phi))$.
That is,
its derivative $\Phi'$ exists a.e. and
\begin{equation}\label{derivative}
 \Phi(t)=\int_0^t\Phi'(s)\,ds, \quad t\in[0,b(\Phi)).
\end{equation}

\begin{defn}\label{defn:iPy}
\begin{enumerate}
\item
Let $\iPy$ be the set of all Young functions.
\item
Let $\biPy$ be the set of all $\Phi\in\biP$ such that
$\Phi\approx\Psi$ for some $\Psi\in \iPy$.
\item
Let $\cY$ be the set of all Young functions such that $a(\Phi)=0$ and $b(\Phi)=\infty$.
\end{enumerate}
\end{defn}

For $\Phi\in\biPy$, we define the Orlicz space $\LP(\R^n)$ and the weak Orlicz space $\wLP(\R^n)$.
Let $L^0(\R^n)$ be the set of all complex valued measurable functions on $\R^n$.
\begin{defn}\label{defn:LP}
For a function $\Phi\in\biPy$, let
\begin{align*}
  \LP(\R^n)
  &= \left\{ f\in L^0(\R^n):
     \int_{\R^n} \Phi(\epsilon |f(x)|)\,dx<\infty
                 \;\text{for some}\; \epsilon>0 
    \right\}, \\
  \|f\|_{\LP} &=
  \inf\left\{ \lambda>0: 
    \int_{\R^n} \!\Phi\!\left(\frac{|f(x)|}{\lambda}\right) dx
      \le 1
      \right\}, \\
  \wLP(\Omega)
  &= \left\{ f\in L^0(\R^n):
     \sup_{t\in(0,\infty)}\Phi(t)\,m(\epsilon f, t)<\infty
                 \;\text{for some}\; \epsilon>0 
    \right\}, \\
  \|f\|_{\wLP} &=
  \inf\left\{ \lambda>0: 
    \sup_{t\in(0,\infty)}\Phi(t)\,m\!\left(\frac{f}{\lambda}, t\right)
      \le 1
      \right\}, \\
  &\text{where} \quad m(f,t)=|\{x\in\R^n:|f(x)|>t\}|.
\end{align*}
\end{defn}
Then 
$\|\cdot\|_{\LP}$ and $\|\cdot\|_{\wLP}$ 
are quasi-norms and $\LP(\R^n)\subset L^1_{\loc}(\R^n)$.
If $\Phi\in\iPy$, then $\|\cdot\|_{\LP}$ is a norm 
and thereby $\LP(\R^n)$ is a Banach space.
For $\Phi,\Psi\in\biPy$,
if $\Phi\approx\Psi$, then $\LP(\R^n)=\LPs(\R^n)$ and $\wLP(\R^n)=\wLPs(\R^n)$
with equivalent quasi-norms, respectively.
Orlicz spaces are introduced by \cite{Orlicz1932,Orlicz1936}.
For the theory of Orlicz spaces,
see \cite{Kita2009,Kokilashvili-Krbec1991,Krasnoselsky-Rutitsky1961,Maligranda1989,Rao-Ren1991}
for example.

We note that, for any Young function $\Phi$, we have that
\begin{equation*}
 \sup_{t\in(0,\infty)}\Phi(t)\,m(f,t)
 =
 \sup_{t\in(0,\infty)}t\,m(\Phi(|f|),t),
\end{equation*}
and then
\begin{align*}
  \|f\|_{\wLP} &=
  \inf\left\{ \lambda>0: 
    \sup_{t\in(0,\infty)}\Phi(t)\,m\!\left(\frac{f}{\lambda}, t\right)
      \le 1
      \right\} \\
  &=
  \inf\left\{ \lambda>0: 
    \sup_{t\in(0,\infty)}t\;m\!\left(\Phi\left(\frac{|f|}{\lambda}\right), t\right)
      \le 1
      \right\}.
\end{align*}
For the above equality, see \cite[Proposition~4.2]{Kawasumi-Nakai-preprint} for example.

\begin{defn}\label{defn:D2 n2}
\begin{enumerate}
\item 
A function $\Phi\in\biP$ is said to satisfy the $\Delta_2$-condition,
denote $\Phi\in\bdtwo$, 
if there exists a constant $C>0$ such that
\begin{equation}\label{Delta2}
 \Phi(2t)\le C\Phi(t) 
 \quad\text{for all } t>0.
\end{equation}
\item
A function $\Phi\in\biP$ is said to satisfy the $\nabla_2$-condition,
denote $\Phi\in\bntwo$, 
if there exists a constant $k>1$ such that
\begin{equation}\label{nabla2}
 \Phi(t)\le\frac1{2k}\Phi(kt) 
 \quad\text{for all } t>0.
\end{equation}
\item
Let $\Delta_2=\iPy\cap\bdtwo$ and $\nabla_2=\iPy\cap\bntwo$.
\end{enumerate}
\end{defn}

\begin{rem}\label{rem:D2 n2}
\begin{enumerate}
\item 
$\dtwo\subset\cY$ and $\bntwo\subset\biPy$ (\cite[Lemma~1.2.3]{Kokilashvili-Krbec1991}).

\item
Let $\Phi\in\biPy$.
Then
$\Phi\in\bdtwo$ if and only if $\Phi\approx\Psi$ for some $\Psi\in\dtwo$,
and,  
$\Phi\in\bntwo$ if and only if $\Phi\approx\Psi$ for some $\Psi\in\ntwo$.

\item
Let $\Phi\in\iPy$.
Then $\Phi\in\dtwo$ if and only if $\Cic(\R^n)$ is dense in $\LP(\R^n)$,
and, $\Phi\in\ntwo$ if and only if the Hardy-Littlewood maximal operator $M$ is bounded on $\LP(\R^n)$.

\item
Let $\Phi\in\iPy$.
Then
$\Phi^{-1}$ satisfies the doubling condition by its concavity,
that is, 
\begin{equation}\label{Phi-1 doubl}
 \Phi^{-1}(u)\le\Phi^{-1}(2u)\le2\Phi^{-1}(u)
 \quad\text{for all $u\in[0,\infty]$}.
\end{equation}
\end{enumerate}
\end{rem}

The following theorem is known,
see \cite[Theorem~1.2.1]{Kokilashvili-Krbec1991} for example.

\begin{thm}\label{thm:M}
Let $\Phi\in\biPy$.
Then 
$M$ is bounded from $\LP(\R^n)$ to $\wLP(\R^n)$,
that is,
there exists a positive constant $C_0$
such that, for all $f\in\LP(\R^n)$,
\begin{equation}\label{C0weak}
 \|Mf\|_{\wLP}\le C_0\|f\|_{\LP}.
\end{equation}
Moreover, if $\Phi\in\bntwo$, then
$M$ is bounded on $\LP(\R^n)$,
that is,
there exists a positive constant $C_0$
such that, for all $f\in\LP(\R^n)$,
\begin{equation}\label{C0}
 \|Mf\|_{\LP}\le C_0\|f\|_{\LP}.
\end{equation}
\end{thm}
See also \cite{Cianchi1999,Kita1996PAMS,Kita1997MathNachr}
for the Hardy-Littlewood maximal operator on Orlicz spaces.

\section{Main results}\label{sec:main}

The following theorem is an extension of the result in \cite{Nakai2001Taiwan}
and has been proved in \cite{Deringoz-Guliyev-Nakai-Sawano-Shi-preprint}
essentially, by using Hedberg's method in \cite{Hedberg1972}.

\begin{thm}[\cite{Deringoz-Guliyev-Nakai-Sawano-Shi-preprint}]\label{thm:Ir}
Let $\rho:(0,\infty)\to(0,\infty)$ satisfy \eqref{int rho} and \eqref{sup rho},
and let $\Phi,\Psi\in\biPy$.
Assume that 
there exists a positive constant $A$ such that,
for all $r\in(0,\infty)$,
\begin{equation}\label{Ir A}
 \int_0^r\frac{\rho(t)}{t}\,dt\;{\Phi}^{-1}(1/r^n) 
  +\int_r^{\infty}\frac{\rho(t)\,\Phi^{-1}(1/t^n)}{t}\,dt
 \le
 A\Psi^{-1}(1/r^n). 
\end{equation}
Then, for any positive constant $C_0$, there exists a positive constant $C_1$ such that, 
for all $f\in L^{\Phi}(\R^n)$ with $f\not\equiv0$,
\begin{equation}\label{Ir pointwise}
     \Psi\left(\frac{|\Ir f(x)|}{C_1\|f\|_{\LP}}\right)
     \le \Phi\left(\frac{Mf(x)}{C_0\|f\|_{\LP}}\right).
\end{equation}
Consequently, $I_{\rho}$ is bounded 
from $L^{\Phi}(\R^n)$ to $\wLPs(\R^n)$. 
Moreover, if $\Phi\in\bntwo$, then
$I_{\rho}$ is bounded 
from $\LP(\R^n)$ to $\LPs(\R^n)$. 
\end{thm}

\begin{rem}\label{rem:Ir}
In \cite{Deringoz-Guliyev-Nakai-Sawano-Shi-preprint}
the condition that $\Phi,\Psi\in\iPy$ was assumed.
We can extend it to $\Phi,\Psi\in\biPy$ as Theorem~\ref{thm:Ir}.
Actually, if \eqref{Ir A} holds for some $\Phi,\Psi\in\biPy$,
then take $\Phi_1,\Psi_1\in\iPy$ with $\Phi\approx\Phi_1$ and $\Psi\approx\Psi_1$.
Then,
instead of $\Phi$ and $\Psi$, 
$\Phi_1$ and $\Psi_1$ satisfy \eqref{Ir A} for some positive constant $A'$
by \eqref{approx equiv}.
\end{rem}

Here, we give some examples of the pair of $(\rho,\Phi,\Psi)$ 
which satisfies the assumption in Theorem~\ref{thm:Ir}.
For other examples, see \cite{Nakai2001SCMJ}.
See also \cite{Mizuta-Nakai-Ohno-Shimomura2010JMSJ} for the boundedness of $\Ir$ 
on Orlicz space $\LP(\Omega)$ with bounded domain $\Omega\subset\R^n$.

\begin{exmp}\label{exmp:Ir1}
If $\rho(r)=r^{\alpha}$, $\Phi(t)=t^p$ and $\Psi(t)=t^q$ with 
$p,q\in[1,\infty)$ and $0<\alpha<n/p$, 
then
\begin{equation*}
 \int_0^r\frac{\rho(t)}{t}\,dt\;{\Phi}^{-1}(1/r^n) 
 \sim \int_r^{\infty}\frac{\rho(t)\,\Phi^{-1}(1/t^n)}{t}\,dt
 \sim r^{\alpha-n/p}
 \quad\text{and}\quad
 \Psi^{-1}(1/r^n)=r^{-n/q}. 
\end{equation*}
In this case,
\begin{equation*}
 \text{``\eqref{Ir A}"}
 \quad\Leftrightarrow\quad 
 r^{\alpha-n/p}\ls r^{-n/q}, \ r\in(0,\infty)
 \quad\Leftrightarrow\quad 
 \alpha-n/p=-n/q.
\end{equation*}
Therefore, the Hardy-Littlewood-Sobolev theorem is a corollary of Theorem~\ref{thm:Ir}.
\end{exmp}

\begin{exmp}\label{exmp:Ir2}
Let $\rho$ and $\Phi$ be as in \eqref{rho log} and in \eqref{Phi exp}, respectively,
and let $\Psi$ be as in \eqref{Phi exp} with $q$ instead of $p$.
Assume that $\alpha,p,q\in(0,\infty)$ and $-1/p+\alpha=-1/q$.
Then
\begin{equation*}
 \int_0^r\frac{\rho(t)}{t}\,dt\sim
 \begin{cases}
 (\log(1/r))^{-\alpha} & \text{for small $r>0$}, \\
 (\log r)^{\alpha} & \text{for large $r>0$},
 \end{cases}
\end{equation*}
and
\begin{equation}\label{Phi Psi}
 \Phi^{-1}(1/r^{n})\sim
 \begin{cases}
 (\log(1/r))^{1/p}, \\
 (\log r)^{-1/p},
 \end{cases}
 \Psi^{-1}(1/r^{n})\sim
 \begin{cases}
 (\log(1/r))^{1/q} & \text{for small $r>0$}, \\
 (\log r)^{-1/q} & \text{for large $r>0$}.
 \end{cases}
\end{equation}
In this case we have
\begin{multline*}
  \int_0^r\frac{\rho(t)}{t}\,dt\;{\Phi}^{-1}(1/r^n)
 \sim
  \int_r^{\infty}\frac{\rho(t)\,\Phi^{-1}(1/t^n)}{t}\,dt \\
 \sim
 \begin{cases}
 (\log(1/r))^{-\alpha+1/p} & \text{for small $r>0$}, \\
 (\log r)^{\alpha-1/p} & \text{for large $r>0$}.
 \end{cases}
\end{multline*}
Then the pair $(\rho,\Phi,\Psi)$ satisfies \eqref{Ir A},
that is, $I_{\rho}$ is bounded from $\exp L^p(\R^n)$ to $\exp L^q(\R^n)$.
\end{exmp}

\begin{exmp}\label{exmp:Ir3}
Let $\alpha\in(0,n)$, $p,q\in[1,\infty)$ and $-n/p+\alpha=-n/q$.
Let
\begin{equation*}
 \rho(r)=
 \begin{cases}
 r^{\alpha} & \text{for small $r>0$}, \\
 e^{-r} & \text{for large $r>0$}.
 \end{cases}
\end{equation*}
Then 
\begin{equation*}
 \int_0^r\frac{\rho(t)}{t}\,dt\sim
 \begin{cases}
 r^{\alpha} & \text{for small $r>0$}, \\
 1 & \text{for large $r>0$}.
 \end{cases}
\end{equation*}
\begin{enumerate}
\item 
If $\Phi(r)=r^p$ and $\Psi(r)=\max(r^p,r^q)$, 
then \eqref{Ir A} holds.
In this case $\LP(\R^n)=L^p(\R^n)$ and $\LPs(\R^n)=L^p(\R^n)\cap L^q(\R^n)$.
\item
If $\Phi(r)=\max(0,r^p-1)$ and $\Psi(r)=\max(0,r^q-1)$, 
then \eqref{Ir A} holds,
since 
\begin{equation*}
 \Phi^{-1}(u)\sim
 \begin{cases}
 1 & \text{for small $u>0$}, \\
 u^{1/p} & \text{for large $u>0$},
 \end{cases}
 \ \
 \Phi^{-1}(1/r^n)\sim
 \begin{cases}
 r^{-n/p} & \text{for small $r>0$}, \\
 1 & \text{for large $r>0$}.
 \end{cases}
\end{equation*}
In this case $\LP(\R^n)=L^p(\R^n)+\Li(\R^n)$ and $\LPs(\R^n)=L^q(\R^n)+\Li(\R^n)$.
\end{enumerate}
\end{exmp}

A function $\Phi\in\cY$ is called an N-function if
\begin{equation*}
 \lim_{t\to+0}\frac{\Phi(t)}t=0,
 \quad
 \lim_{t\to\infty}\frac{\Phi(t)}t=\infty.
\end{equation*}
We say that a function $\theta:(0,\infty)\to(0,\infty)$ 
is almost increasing (resp. almost decreasing) if
there exists a positive constant $C$ such that, for all $r,s\in(0,\infty)$,
\begin{equation}\label{almost inc dec}
 \theta(r)\le C\theta(s) \quad
 (\text{resp.}\ \theta(s)\le C\theta(r)),
 \quad\text{if $r<s$}.
\end{equation}
Then we have the following corollary.

\begin{cor}\label{cor:Ir}
Let $1<s<\infty$ and $\rho:(0,\infty)\to(0,\infty)$.
Assume that $\rho$ satisfies \eqref{int rho}
and that 
$r\mapsto\rho(r)/r^{n/s-\epsilon}$ is almost decreasing for some positive constant $\epsilon$.
Then there exist an N-function $\Psi$ and a positive constant $C$ such that,
for all $r>0$,
\begin{equation}\label{s Psi}
 C^{-1}\Psi^{-1}\left(\frac1{r^n}\right)
 \le
 \frac1{r^{n/s}}\int_0^r\frac{\rho(t)}{t}\,dt
 \le
 C\Psi^{-1}\left(\frac1{r^n}\right).
\end{equation}
Moreover, $\Ir$ is bounded from $L^s(\R^n)$ to $\LPs(\R^n)$.
\end{cor}

In the above, \eqref{s Psi} can be shown by
the same way as the proof of \cite[Theorem~3.5]{Arai-Nakai2017REMC}.
The boundedness of $\Ir$ from $L^s(\R^n)$ to $\LPs(\R^n)$ is proven by the following way.
First note that $\rho$ satisfies \eqref{sup rho} by Remark~\ref{rem:cor:Ir} below.
Let $\Phi(t)=t^s$.
Then we have
\begin{align*}
 \int_r^{\infty}\frac{\rho(t)\Phi^{-1}(1/t^n)}{t}\,dt
 &=
 \int_r^{\infty}\frac{\rho(t)/t^{n/s}}{t}\,dt
 \ls
 \frac{\rho(r)}{r^{n/s-\epsilon}}\int_r^{\infty}\frac{1}{t^{1+\epsilon}}\,dt \\
 &\sim
 \frac{\rho(r)}{r^{n/s}} 
 \ls
 \frac1{r^{n/s}}\int_0^r\frac{\rho(t)}{t}\,dt
 =
 \Phi^{-1}\left(\frac1{r^n}\right)\int_0^r\frac{\rho(t)}{t}\,dt,
\end{align*}
where we used \eqref{rho/r^k} below for the last inequality.
Combining this and \eqref{s Psi},
we have \eqref{Ir A}.
Then we have the conclusion by Theorem~\ref{thm:Ir}.

\begin{rem}\label{rem:cor:Ir}
If $r\mapsto\rho(r)/r^k$ is almost decreasing for some positive constant $k$,
then $\rho$ satisfies \eqref{sup rho}.
Actually,
\begin{equation}\label{rho/r^k}
 \sup_{r\le t\le2r}\rho(t)
 \sim
 r^k\sup_{r\le t\le2r}\frac{\rho(t)}{t^k}
 \ls
 r^k\int_{r/2}^r\frac{\rho(t)}{t^{k+1}}\,dt
 \sim
 \int_{r/2}^r\frac{\rho(t)}{t}\,dt.
\end{equation}
\end{rem}

Next we state the result on the operator $\Mr$ defined by \eqref{Mr}
in which we don't assume \eqref{int rho} or \eqref{sup rho}.

\begin{thm}\label{thm:Mr}
Let $\rho:(0,\infty)\to(0,\infty)$,
and let $\Phi,\Psi\in\biPy$.
\begin{enumerate}
\item 
Assume that 
there exists a positive constant $A$ such that,
for all $r\in(0,\infty)$,
\begin{equation}\label{Mr A}
 \left(\sup_{0<t\le r}\rho(t)\right){\Phi}^{-1}(1/r^n)
 \le
 A\Psi^{-1}(1/r^n). 
\end{equation}
Then, for any positive constant $C_0$, there exists a positive constant $C_1$ such that, 
for all $f\in L^{\Phi}(\R^n)$ with $f\not\equiv0$,
\begin{equation}\label{Mr pointwise}
     \Psi\left(\frac{\Mr f(x)}{C_1\|f\|_{\LP}}\right)
     \le \Phi\left(\frac{Mf(x)}{C_0\|f\|_{\LP}}\right).
\end{equation}
Consequently, $\Mr$ is bounded 
from $L^{\Phi}(\R^n)$ to $\wLPs(\R^n)$. 
Moreover, if $\Phi\in\bntwo$, then
$\Mr$ is bounded 
from $\LP(\R^n)$ to $\LPs(\R^n)$. 

\item
Conversely, 
if $\Mr$ is bounded from $\LP(\R^n)$ to $\wLPs(\R^n)$,
then
\eqref{Mr A} holds for some $A$ and all $r\in(0,\infty)$.
\end{enumerate}
\end{thm}

\begin{rem}\label{rem:Mr}
Let $\rho:(0,\infty)\to(0,\infty)$,
and let $\Phi,\Psi\in\biPy$.
\begin{enumerate}
\item
Let $\rho_1(r)=\sup_{0<t\le r}\rho(t)$.
Then we conclude from the theorem above that 
$\Ir$ and $I_{\rho_1}$ have the same boundedness,
that is, we may assume that $\rho$ is increasing.

\item 
Since $\Phi^{-1}$ is pseudo-concave,
$u\mapsto\Phi^{-1}(u)/u$ is almost decreasing, and then
$r\mapsto\Phi^{-1}(1/r^n)r^n$ is almost increasing.
Therefore, from \eqref{Mr A} it follows that 
$r\mapsto\rho(r)/r^n$ is dominated by the almost decreasing function
$r\mapsto\frac{\Psi^{-1}(1/r^n)}{\Phi^{-1}(1/r^n)r^n}$.

\item
In \cite{Deringoz-Guliyev-Nakai-Sawano-Shi-preprint},
under the conditions that $\Phi,\Psi\in\iPy$, that $\rho$ is increasing and that
$r\mapsto\rho(r)/r^n$ is decreasing,
a necessary and sufficient condition for the boundedness of $\Mr$ has been given.
\end{enumerate}
\end{rem}

\begin{exmp}\label{exmp:Mr1}
If $\rho(r)=r^{\alpha}$, $\Phi(t)=t^p$ and $\Psi(t)=t^q$ with
$p,q\in[1,\infty)$ and $0\le\alpha\le n/p$, 
then
\begin{equation*}
 \rho(r){\Phi}^{-1}(1/r^n) 
 \sim r^{\alpha-n/p}
 \quad\text{and}\quad
 \Psi^{-1}(1/r^n)=r^{-n/q}. 
\end{equation*}
In this case,
\begin{equation*}
 \text{``\eqref{Mr A}"}
 \quad\Leftrightarrow\quad 
 r^{\alpha-n/p}\ls r^{-n/q}, \ r\in(0,\infty)
 \quad\Leftrightarrow\quad 
 \alpha-n/p=-n/q.
\end{equation*}
In this example, if $\alpha=0$, then $\Mr$ is the Hardy-Littlewood maximal operator $M$
and $\text{``\eqref{Mr A}"}\Leftrightarrow p=q$.
If $\alpha-n/p=0$, then $\Mr$ is the fractional maximal operator $\Ma$
and it is bounded from $L^p(\R^n)$ to $\Li(\R^n)$,
since we can take 
\begin{equation}\label{Psi Psi-1}
 \Psi(r)=
 \begin{cases}
 0 & \text{for $r\in[0,1]$}, \\
 \infty &\text{for $r\in(1,\infty]$},
 \end{cases}
 \quad\text{and}\quad
 \Psi^{-1}(r)=
 \begin{cases}
 1 & \text{for $r\in[0,\infty)$}, \\
 \infty &\text{for $r=\infty$}.
 \end{cases}
\end{equation}
\end{exmp}

\begin{exmp}\label{exmp:Mr2}
Let $\Phi$ be as in \eqref{Phi exp},
and let $\Psi$ be as in \eqref{Phi exp} with $q$ instead of $p$.
Assume that $\alpha\in[0,\infty)$ and $p,q\in(0,\infty)$.
Let
\begin{equation}\label{M rho}
 \rho(r)=
 \begin{cases}
 (\log(1/r))^{-\alpha} & \text{for small $r>0$}, \\
 (\log r)^{\alpha} & \text{for large $r>0$},
 \end{cases}
\end{equation}
instead of \eqref{rho log}.
Here, we note that,
if $0\le\alpha\le1$, then $\int_0^1\frac{\rho(t)}{t}\,dt=\infty$,
that is, $\Ir$ is not well defined, while $\Mr$ is well defined.
Actually, $\Mr$ is bounded from $\exp L^p(\R^n)$ to $\exp L^q(\R^n)$,
if $-1/p+\alpha=-1/q$ for any $\alpha\in[0,\infty)$,
see \eqref{Phi Psi} for the inverse functions of $\Phi$ and $\Psi$.
Moreover, 
if $-1/p+\alpha=0$, then $\Mr$ is bounded from $\exp L^p(\R^n)$ to $\Li(\R^n)$,
since we can take $\Psi$ as in \eqref{Psi Psi-1}.
\end{exmp}

\begin{exmp}\label{exmp:Mr3}
Assume that $\alpha,q\in[0,\infty)$ and $p\in(1,\infty)$.
Let $\rho$ be as in \eqref{M rho}.
Then
$\Mr$ is bounded from $L^p(\R^n)$ to $L^p(\log L)^{p_1}(\R^n)$,
if $p_1/p=\alpha$,
where $L^p(\log L)^{p_1}(\R^n)$ is the Orlicz space $\LP(\R^n)$ with
\begin{equation*}
 \Phi(r)=
 \begin{cases}
 r^{p}(\log(1/r))^{-p_1} & \text{for small $r>0$}, \\
 r^{p}(\log r)^{p_1} &\text{for large $r>0$}.
 \end{cases}
\end{equation*}
In this case we have
\begin{equation}\label{Phi-1}
 \Phi^{-1}(1/r^n) \sim
 \begin{cases}
 r^{-n/p}(\log(1/r))^{-p_1/p} & \text{for small $r>0$}, \\
 r^{-n/p}(\log r)^{p_1/p} & \text{for large $r>0$}.
 \end{cases}
\end{equation}
In this example, if we take $p=1$, then 
$\Mr$ is bounded from $L^1(\R^n)$ to $\wL^1(\log L)^{\alpha}(\R^n)$
which is weak type of $L^1(\log L)^{\alpha}(\R^n)$.
\end{exmp}

Finally, we state the result on the commutator $[b,\Ir]$.
Let
\begin{equation}\label{rho*}
 \rho^*(r)=\int_0^r\frac{\rho(t)}{t}\,dt.
\end{equation}

\begin{thm}\label{thm:comm Ir}
Let $\rho,\psi:(0,\infty)\to(0,\infty)$,
and let $\Phi,\Psi\in\biPy$.
Assume that $\rho$ satisfies \eqref{int rho}.
Let $b\in L^1_{\loc}(\R^n)$.
\begin{enumerate}
\item 
Let $\Phi,\Psi\in\bdtwo\cap\bntwo$.
Assume that $\psi$ be almost increasing
and that $r\mapsto\rho(r)/r^{n-\epsilon}$ 
is almost decreasing for some $\epsilon\in(0,n)$.
Assume also that 
there exists a positive constant $A$ and $\Theta\in\bntwo$ such that,
for all $r\in(0,\infty)$,
\begin{gather}\label{comm Ir A}
 \int_0^r\frac{\rho(t)}{t}\,dt\;{\Phi}^{-1}(1/r^n) 
  +\int_r^{\infty}\frac{\rho(t)\,\Phi^{-1}(1/t^n)}{t}\,dt
 \le
 A\Theta^{-1}(1/r^n),
 \\
 \label{comm Mr A}
 \psi(r)\Theta^{-1}(1/r^n)
 \le A
 \Psi^{-1}(1/r^n),
\end{gather}
and that
there exist a positive constant $C_{\rho}$ 
such that,
for all $r,s\in(0,\infty)$,
\begin{equation}\label{rho conti} 
 \left|\frac{\rho(r)}{r^n}-\frac{\rho(s)}{s^n}\right| 
  \le
  C_{\rho}\,|r-s|\frac{\rho^*(r)}{r^{n+1}},
 \quad \text{if $\frac12\le\frac rs\le2$}.
\end{equation}
If $b\in\cL_{1,\psi}(\R^n)$,
then
$[b,\Ir]$ is bounded from $\LP(\R^n)$ to $\LPs(\R^n)$
and there exists a positive constant $C$ such that,
for all $f\in\LP(\R^n)$,
\begin{equation}\label{comm Ir}
 \|[b,\Ir]f\|_{\LPs}\le C\|b\|_{\cL_{1,\psi}}\|f\|_{\LP}.
\end{equation}

\item
Conversely, 
assume that there exists a positive constant $A$ such that,
for all $r\in(0,\infty)$,
\begin{equation*}
 \Psi^{-1}(1/r^{n})
 \le 
 Ar^{\alpha}\psi(r)\Phi^{-1}(1/r^{n}).
\end{equation*}
If $[b,\Ia]$ is well defined and bounded from $\LP(\R^n)$ to $\LPs(\R^n)$,
then $b$ is in $\cL_{1,\psi}(\R^n)$ 
and there exists a positive constant $C$, independent of $b$, such that
\begin{equation}\label{comm Ir c}
 \|b\|_{\cL_{1,\psi}}\le C\|[b,\Ia]\|_{\LP\to\LPs},
\end{equation}
where $\|[b,\Ia]\|_{\LP\to\LPs}$ is the operator norm of $[b,\Ia]$ from $\LP(\R^n)$ to $\LPs(\R^n)$.
\end{enumerate}
\end{thm}

\begin{exmp}\label{exmp:comm0}
Let $\alpha\in(0,n)$, $\beta\in[0,1]$ and $p,q\in(1,\infty)$,
and, let
\begin{equation*}
 \rho(r)=r^{\alpha}, \ \psi(r)=r^{\beta}, \ \Phi(r)=r^p, \ \Psi(r)=r^q.
\end{equation*}
Assume that $-n/p+\alpha+\beta=-n/q$.
Take $\Theta(r)=r^{\tq}$ with $-n/\tq=-n/p+\alpha$.
Then \eqref{comm Ir A}, \eqref{comm Mr A} and \eqref{rho conti} hold,
that is, $[b,\Ia]$ is bounded from $L^p(\R^n)$ to $L^q(\R^n)$, 
where $b\in\Lip_{\beta}(\R^n)$ if $\beta\in(0,1]$, 
and $b\in\BMO(\R^n)$ if $\beta=0$ which is Chanillo's result in \cite{Chanillo1982}.
\end{exmp}

\begin{exmp}\label{exmp:comm}
Let $\alpha\in(0,n)$ and $\alpha_1\in(-\infty,\infty)$.
Let $\beta\in(0,n)$ and $\beta_1\in(-\infty,\infty)$, or,
let $\beta=0$ and $\beta_1\in[0,\infty)$.
Let
\begin{equation*}
 \rho(r)=
 \begin{cases}
 r^{\alpha}(\log(1/r))^{-\alpha_1}, \\
 r^{\alpha}, \\
 r^{\alpha}(\log r)^{\alpha_1},
 \end{cases}
 \psi(r)=
 \begin{cases}
 r^{\beta}(\log(1/r))^{-\beta_1} & \text{for $r\in(0,1/e)$}, \\
 r^{\beta} & \text{for $r\in[1/e,e]$}, \\
 r^{\beta}(\log r)^{\beta_1} & \text{for $r\in(e,\infty)$}.
 \end{cases}
\end{equation*}
Then $\rho^*\sim\rho$ and $\rho'(t)\sim\rho(t)/t$.
In this case $\rho$ satisfies \eqref{rho conti},
since $\rho$ is Lipschitz continuous on $[1/(2e),2e]$,
and, for $r,s\in (0,1/e]\cup[e,\infty)$, there exists $\theta\in(0,1)$ such that
\begin{equation*}
 \left|\frac{\rho(r)}{r^n}-\frac{\rho(s)}{s^n}\right| 
 =
 |r-s|\left|\left.\frac{d}{dt}\left(\frac{\rho(t)}{t^n}\right)\right|_{t=(1-\theta)r+\theta s}\right|
 \ls
 |r-s|\frac{\rho(r)}{r^{n+1}},
 \quad \text{if $\frac12\le\frac rs\le2$}.
\end{equation*}
Let $p,q\in(1,\infty)$ and $p_1,q_1\in(-\infty,\infty)$, and let
\begin{equation*}
 \Phi(r)=
 \begin{cases}
 r^{p}(\log(1/r))^{-p_1}, \\
 r^{p}(\log r)^{p_1},
 \end{cases}
 \Psi(r)=
 \begin{cases}
 r^{q}(\log(1/r))^{-q_1} & \text{for small $r>0$}, \\
 r^{q}(\log r)^{q_1} &\text{for large $r>0$}.
 \end{cases}
\end{equation*}
For the inverse functions of $\Phi$ and $\Psi$, see \eqref{Phi-1}.
If 
\begin{equation*}
 -n/p+\alpha+\beta=-n/\tp+\beta=-n/q,
 \quad
 p_1/p+\alpha_1+\beta_1=\tp_1/\tp+\beta_1=q_1/q,
\end{equation*}
and
\begin{equation*}
 \Theta(r)=
 \begin{cases}
 r^{\tp}(\log(1/r))^{-\tp_1} & \text{for small $r>0$}, \\
 r^{\tp}(\log r)^{\tp_1} &\text{for large $r>0$},
 \end{cases}
\end{equation*}
then
\begin{equation*}
  \int_0^r\frac{\rho(t)}{t}\,dt\;{\Phi}^{-1}(1/r^n) \sim
  \int_r^{\infty}\frac{\rho(t)\,\Phi^{-1}(1/t^n)}{t}\,dt \sim
  \Theta^{-1}(r^{-n}),
\end{equation*}
and
\begin{equation*}
\psi(r) \Theta^{-1} (r^{-n}) \sim \Psi^{-1} (r^{-n}) \sim
 \begin{cases}
 r^{-n/p+\alpha+\beta}(\log(1/r))^{-(p_1/p+\alpha_1+\beta_1)} & \text{for small $r>0$}, \\
 r^{-n/p+\alpha+\beta}(\log r)^{p_1/p+\alpha_1+\beta_1} & \text{for large $r>0$}.
 \end{cases}
\end{equation*}
In this case $[b,\Ir]$ is bounded from $L^p(\log L)^{p_1}(\R^n)$ to $L^q(\log L)^{q_1}(\R^n)$.
\end{exmp}

\section{Lemmas}\label{sec:lemmas}

In this section we prepare some lemmas to prove our main results.

For a Young function $\Phi$, 
its complementary function is defined by
\begin{equation*}\label{complementary}
\cPhi(t)= 
\begin{cases}
   \sup\{tu-\Phi(u):u\in[0,\infty)\}, & t\in[0,\infty), \\
   \infty, & t=\infty.
 \end{cases}
\end{equation*}
Then $\cPhi$ is also a Young function and Young's inequality 
\begin{equation*}\label{Young's_ineq}
 tu\le\Phi(t)+\cPhi(u),
 \quad
 t,u\in[0,\infty)
\end{equation*}
holds.
It is also known that
\begin{equation}\label{PhicPhi r}
t\le\Phi^{-1}(t) \cPhi^{-1}(t)\le2t, \quad t\ge 0.
\end{equation}
From Young's inequality we have a generalized H\"older's inequality:
\begin{equation}\label{g Holder}
 \int_{\mathbb R^n}  |f(x)g(x)|\,dx \le 2\|f\|_{\LP} \|g\|_{\LcP}
\end{equation}
(see \cite[Theorem~6]{Weiss1956} and \cite[Theorem~2.3]{ONeil1965}).

\begin{lem}\label{lem:chi}
Let $\Phi\in\iPy$.
For a measurable set $G\subset\R^n$ with finite measure,
\begin{equation*}
 \|\chi_G\|_{\LP}
 =
 \|\chi_G\|_{\wLP}
 =
 \frac1{\Phi^{-1}(1/|G|)}.
\end{equation*}
\end{lem}

From \eqref{PhicPhi r} it follows that, for the characteristic function $\chi_B$ of the ball $B$,
\begin{equation}\label{chi norm}
 \|\chi_B\|_{\LcP}=\frac1{\cPhi^{-1}(1/|B|)}\le|B|\Phi^{-1}(1/|B|).
\end{equation}

\begin{lem}[\cite{Arai-Nakai2017REMC}]\label{lem:rho dec}
Let $k>0$ and $\rho:(0,\infty)\to(0,\infty)$.
Assume that $\rho$ satisfies \eqref{int rho}.
Let $\rho^*$ be as in \eqref{rho*}.
If $r\mapsto\rho(r)/r^k$ is almost decreasing,
then
$r\mapsto\rho^*(r)/r^k$ is also almost decreasing.
\end{lem}

\begin{rem}\label{rem:rho dec}
Since $\rho^*$ is increasing with respect to $r$, 
if $r\mapsto\rho(r)/r^k$ is almost decreasing for some $k>0$,
then we see that $\rho^*$ satisfies the doubling condition, 
that is, there exists a positive constant $C$ such that, for all $r\in(0,\infty)$,
\begin{equation*}
 \rho^*(r)\le\rho^*(2r)\le C\rho^*(r).
\end{equation*}
\end{rem}

\begin{lem}\label{lem:diff D2}
If $\Phi\in\dtwo$,
then its derivative $\Phi'$ satisfies
\begin{equation*}
 \Phi'(2t)\le C_{\Phi}\Phi'(t), 
 \quad\text{a.e.}\, t\in[0,\infty),
\end{equation*}
where the constant $C_{\Phi}$ is independent of $t$.
\end{lem}

\begin{proof}
From the convexity of $\Phi$ and $\Phi(0)=0$
it follows that its right derivative 
$\Phi_+'(t)$ exists for all $t\in[0,\infty)$ and it is increasing.
By \eqref{derivative} we have
\begin{equation*}
 \Phi(t)=\int_0^t\Phi'(s)\,ds=\int_0^t\Phi_+'(s)\,ds,
\end{equation*}
since $\Phi'=\Phi_+'$ a.e.
Then, for all $t\in(0,\infty)$,
\begin{equation*}
 \Phi_+'(2t)
 \le
 \frac{1}{t}\int_{2t}^{3t}\Phi_+'(s)\,ds
 \le
 \frac{1}{t}\Phi(3t)
 \le
 \frac{C_{\Phi}}{t}\Phi(t)
 \le
 C_{\Phi}\Phi_+'(t).
\end{equation*}
This shows the conclusion.
\end{proof}

\begin{lem}\label{lem:eta}
If $\Phi\in\bntwo$, 
then $\Phi((\cdot)^{\theta})\in\bntwo$ for some $\theta\in(0,1)$.
\end{lem}

\begin{proof}
If $\Phi\in\bntwo$, 
then there exists a constant $k>1$ such that 
\begin{equation*}
 \Phi(t)\le\frac1{2k}\Phi(kt).
\end{equation*}
Take $\theta\in(0,1)$ such that $k^{2(1/\theta-1)}\le2$.
Then $k^2\le(2k^2)^{\theta}$ and
\begin{equation*}
 \Phi(t^{\theta})
 \le
 \frac1{2k}\Phi(kt^{\theta}) 
 \le
 \frac1{(2k)^2}\Phi(k^2t^{\theta})
 \le
 \frac1{2(2k^2)}\Phi((2k^2t)^{\theta}).
\end{equation*}
That is, $\Phi((\cdot)^{\theta})\in\bntwo$.
\end{proof}

\begin{rem}\label{rem:eta}
There exists $\Phi\in\ntwo$ such that
$\Phi((\cdot)^{\theta})\notin\iPy$ for any $\theta\in(0,1)$.
Actually,
let
\begin{equation*}
 \Phi(r)=\max(r^2,3r-2)=
 \begin{cases}
  r^2, & 0\le r\le1, \\
  3r-2, & 1<r<2, \\
  r^2, & 2\le r.
 \end{cases}
\end{equation*}
Then $\Phi$ is convex and satisfies \eqref{nabla2} with $k=8$.
However, $3r^{\theta}-2$ is not convex for any $\theta\in(0,1)$.
\end{rem}

\section{Proof of Theorem~\ref{thm:Mr}}\label{sec:proofs bdd}

In this section we prove Theorem~\ref{thm:Mr}.

\begin{proof}[Proof of Theorem~\ref{thm:Mr} {\rm (i)}]
We may assume that $\Phi,\Psi\in\iPy$ by \eqref{approx equiv}.
Let $f\in\LP(\R^n)$.
We may also assume that $\|f\|_{\LP}=1$ and $Mf(x)>0$ for all $x\in\R^n$.
For any $x\in\R^n$ and any ball $B=B(z,r)\ni x$,
if
\begin{equation*}
 \Phi\left(\frac{Mf(x)}{C_0}\right)\ge\frac{1}{r^n},
\end{equation*}
then,
by \eqref{g Holder}, $\|f\|_{\LP}=1$, \eqref{chi norm},
the doubling condition of $\Phi^{-1}$ and \eqref{Mr A},
we have
\begin{align*}
 \rho(r)\mint_B|f|
 &\le
 2\frac{\rho(r)}{|B|}\|\chi_B\|_{\LcP} 
 \le
 2\frac{\rho(r)}{|B|}|B|\Phi^{-1}\left(\frac{1}{|B|}\right) \\
 &\ls
 \rho(r)\Phi^{-1}\left(\frac{1}{r^n}\right) 
 \le
 A\Psi^{-1}\left(\frac{1}{r^n}\right)
 \le
 A\Psi^{-1}\left(\Phi\left(\frac{Mf(x)}{C_0}\right)\right).
\end{align*}
Conversely, if
\begin{equation*}
 \Phi\left(\frac{Mf(x)}{C_0}\right)\le\frac{1}{r^n},
\end{equation*}
then, choosing $t_0\ge r$ such that
\begin{equation*}
 \Phi\left(\frac{Mf(x)}{C_0}\right)=\frac{1}{{t_0}^n},
\end{equation*}
and using \eqref{Mr A} and \eqref{inverse ineq}, we have
\begin{equation*}
 \rho(r)
 \le
 \sup_{0<t\le t_0}\rho(t)
 \le A
 \frac{\Psi^{-1}\left(\Phi\left(\frac{Mf(x)}{C_0}\right)\right)}
      {\Phi^{-1}\left(\Phi\left(\frac{Mf(x)}{C_0}\right)\right)}
 \le A
 \frac{\Psi^{-1}\left(\Phi\left(\frac{Mf(x)}{C_0}\right)\right)}
      {\frac{Mf(x)}{C_0}},
\end{equation*}
which implies 
\begin{equation*}
 \rho(r)\mint_B|f|
 \le AC_0
 \frac{\Psi^{-1}\left(\Phi\left(\frac{Mf(x)}{C_0}\right)\right)}
      {Mf(x)}
  \mint_B|f|
 \le AC_0
 \Psi^{-1}\left(\Phi\left(\frac{Mf(x)}{C_0}\right)\right).
\end{equation*}
Hence, we have
\begin{equation*}
 \Mr f(x)\le C_1\Psi^{-1}\left(\Phi\left(\frac{Mf(x)}{C_0}\right)\right),
\end{equation*}
which shows \eqref{Mr pointwise} by \eqref{inverse ineq}.
\end{proof}

To prove Theorem~\ref{thm:Mr} (ii)
we need the following lemma.
\begin{lem}\label{lem:Mr chi}
Let $\rho:(0,\infty)\to(0,\infty)$.
Then, 
for all $x\in\R^n$ and $r\in(0,\infty)$,
\begin{equation}\label{sup rho < Mr}
 \left(\sup_{0<t\le r}\rho(t)\right)\chi_{B(0,r)}(x)
 \le (\Mr\chi_{B(0,r)})(x).
\end{equation}
\end{lem}

\begin{proof}
Let $x\in B(0,r)$.
If $t\le r$, then we can choose a ball $B(z,t)$ such that $x\in B(z,t)\subset B(0,r)$.
Hence,
\begin{align*}
 \rho(t)=\rho(t)\mint_{B(z,t)}\chi_{B(0,r)}(y)\,dy
 \le
 (\Mr\chi_{B(0,r)})(x).
\end{align*}
Therefore, we have \eqref{sup rho < Mr}.
\end{proof}

\begin{proof}[Proof of Theorem~\ref{thm:Mr} {\rm (ii)}]
By Lemma~\ref{lem:Mr chi} and the boundedness of $\Mr$ from $\LP(\R^n)$ to $\wLPs(\R^n)$
we have
\begin{equation*}
 \left(\sup_{0<t\le r}\rho(t)\right)\|\chi_{B(0,r)}\|_{\wLPs}
 \le
 \|\Mr\chi_{B(0,r)}\|_{\wLPs}
 \ls
 \|\chi_{B(0,r)}\|_{\LP}.
\end{equation*}
Then, by Lemma~\ref{lem:chi} and the doubling condition of $\Phi^{-1}$ and $\Psi^{-1}$
we have the conclusion.
\end{proof}

\section{Sharp maximal operators}\label{sec:sharp}

In this section,
to prove Theorem~\ref{thm:comm Ir}, we prove 
two propositions involving the sharp maximal operator $\Ms$ defined by \eqref{Ms}.

First we state the John-Nirenberg type theorem for the Campanato space,
which is known by \cite[Theorem~3.1]{Nakai2008AMS} for spaces of homogeneous type.
See also \cite{Arai-Nakai2017REMC} for its proof in the case of $\R^n$.

\begin{thm}\label{thm:J-N}
Let $p\in(1,\infty)$ and $\psi:(0,\infty)\to(0,\infty)$.
Assume that $\psi$ is almost increasing.
Then $\cL_{p,\psi}(\R^n)=\cL_{1,\psi}(\R^n)$
with equivalent norms.
\end{thm}

\begin{prop}\label{prop:pointwise Ir}
Assume that $\rho:(0,\infty)\to(0,\infty)$ satisfies \eqref{int rho}.
Let $\rho^*(r)$ be as in \eqref{rho*}.
Assume 
that 
$\psi$ is almost increasing,
that
$r\mapsto\rho(r)/r^{n-\epsilon}$ is almost decreasing for some $\epsilon>0$
and that 
the condition \eqref{rho conti} holds.
Then, for any $\eta\in(1,\infty)$, there exists a positive constant $C$ such that, 
for all $b\in\cL_{1,\psi}(\R^n)$, $f\in\Cic(\R^n)$ and $x\in\R^n$,
\begin{equation}\label{pointwise Ir}
 \Ms([b,\Ir]f)(x)
 \le
 C\|b\|_{\cL_{1,\psi}}
 \bigg(
  \big(M_{\psi^\eta}(|\Ir f|^{\eta})(x)\big)^{1/\eta}
  +\big(M_{(\rho^*\psi)^{\eta}}(|f|^{\eta})(x)\big)^{1/\eta}
 \bigg).
\end{equation}
\end{prop}

To prove the proposition we need the following known lemma, 
for its proof, see Lemma~4.7 and Remark~4.1 in \cite{Arai-Nakai2017REMC} for example.

\begin{lem}\label{lem:int f-fB}
Let $p\in[1,\infty)$. 
Assume that $\psi$ is almost increasing.
Then there exists a positive constant $C$ such that,
for all $f\in\cL_{1,\psi}$, $x\in\R^n$ and $r,s\in(0,\infty)$,
\begin{equation*}
 \left(\mint_{B(x,s)}|f(y)-f_{B(x,r)}|^p\,dy\right)^{1/p}
 \le
 C\left(1+\log_2\frac sr\right)\psi(s)\,\|f\|_{\cL_{1,\psi}},
 \quad\text{if} \ \ r\le s.
\end{equation*}
\end{lem}

\begin{proof}[Proof of Proposition~\ref{prop:pointwise Ir}]
For any ball $B=B(x,t)$, let $f=f_1+f_2$ with $f_1=f\chi_{2B}$, 
and let
\begin{align*}
 F_1(y)&=(b(y)-b_{2B})\Ir f(y), \\
 F_2(y)&=\Ir((b-b_{2B})f_1)(y), \\
 F_3(y)&=\Ir((b-b_{2B})f_2)(y) - C_B,
\end{align*}
for $y\in B$,
where $C_B=\Ir((b-b_{2B})f_2)(x)$ and
\begin{equation*}
 \Ir((b-b_{2B})f_2)(y)
 =
 \int_{\R^n}
 \frac{\rho(|y-z|)}{|y-z|^n}(b(z)-b_{2B})f_2(z)\,dz,
 \quad
 y\in B.
\end{equation*}
Then we have
\begin{equation*}
 [b,\Ir]f + C_B
 = [b-b_{2B},\Ir]f + C_B
 = F_1-F_2-F_3.
\end{equation*}
We show that
\begin{multline}\label{B mean}
 \mint_B|F_i(y)|\,dy \\
 \le
 C\|b\|_{\cL_{1,\psi}} 
 \bigg(
  \big(M_{\psi^\eta}(|\Ir f|^{\eta})(x)\big)^{1/\eta}
  +\big(M_{(\rho^*\psi)^{\eta}}(|f|^{\eta})(x)\big)^{1/\eta}
 \bigg),
 \quad i=1,2,3.
\end{multline}
Then we have the conclusion.

Now, by H\"older's inequality with $1/\eta+1/\eta'=1$ and Theorem~\ref{thm:J-N} we have
\begin{align*}
 \mint_B|F_1(y)|\,dy 
 &\le
 \left(\mint_B|b(y)-b_{2B}|^{\eta'}\,dy\right)^{1/\eta'}
 \left(\mint_B|\Ir f(y)|^{\eta}\,dy\right)^{1/\eta} \\
 &=
 \frac1{\psi(t)}\left(\mint_B|b(y)-b_{2B}|^{\eta'}\,dy\right)^{1/\eta'}
 \left(\psi(t)^\eta\mint_B|\Ir f(y)|^{\eta}\,dy\right)^{1/\eta} \\
 &\ls
 \|b\|_{\cL_{1,\psi}} \big(M_{\psi^\eta}(|\Ir f|^{\eta})(x)\big)^{1/\eta}.
\end{align*}
Choose $v\in(1,\eta)$ such that $n/v-\epsilon/2\ge n-\epsilon$.
Then by the almost decreasingness of $r\mapsto\rho(r)/r^{n-\epsilon}$
we have the almost decreasingness of $r\mapsto\rho(r)/r^{n/v-\epsilon/2}$.
Hence, from Corollary~\ref{cor:Ir} it follows that
there exists an N-function $\Psi$ such that 
$I_{\rho}$ is bounded from $L^v(\R^n)$ to $L^{\Psi}(\R^n)$.
Let $\cPsi$ be the complementary function of $\Psi$.
Then by the generalized H\"older's inequality \eqref{g Holder}, 
\eqref{chi norm}, \eqref{s Psi} and the boundedness of $I_{\rho}$ 
we have
\begin{align*}
 \mint_B|F_2(y)|\,dy
 &\le
 \frac2{|B|}\|\chi_B\|_{\LcPs(\R^n)}\|F_2\|_{\LPs(\R^n)} \\
 &\ls
 \Psi^{-1}(1/|B|)\|(b-b_{2B})f_1\|_{L^v(\R^n)} \\
 &\ls
 \frac{\rho^*(t)}{|B|^{1/v}}\|(b-b_{2B})f\|_{L^v(2B)}.
\end{align*}
Let $1/v=1/u+1/\eta$. Then by H\"older's inequality and Theorem~\ref{thm:J-N} we have
\begin{align*}
 &\mint_B|F_2(y)|\,dy \\
 &\ls 
 \rho^*(t)\left(\mint_{2B}|b(y)-b_{2B}|^u\,dy\right)^{1/u}
 \left(\mint_{2B}|f(y)|^{\eta}\,dy\right)^{1/\eta} \\
 &\ls
 \frac1{\psi(2t)}\left(\mint_{2B}|b(y)-b_{2B}|^u\,dy\right)^{1/u}
 \left((\rho^*(2t)\psi(2t))^{\eta}\mint_{2B}|f(y)|^{\eta}\,dy\right)^{1/\eta} \\
 &\ls
 \|b\|_{\cL_{1,\psi}} \big(M_{(\rho^*\psi)^{\eta}}(|f|^{\eta})(x)\big)^{1/\eta}.
\end{align*}
Finally, using the relation
\begin{equation*}
 \frac12\le\frac{|y-z|}{|x-z|}\le2 
 \quad\text{for $y\in B$ and $z\notin2B$}
\end{equation*}
and \eqref{rho conti}, we have
\begin{align*}
 |F_3(y)|
 &=|\Ir((b-b_{2B})f_2)(y) - \Ir((b-b_{2B})f_2)(x)| \\
 &=
  \left|
   \int_{\R^n}
   \left(\frac{\rho(|y-z|)}{|y-z|^n}-\frac{\rho(|x-z|)}{|x-z|^n}\right)(b(z)-b_{2B})f_2(z)\,dz
  \right| \\
 &\ls
  \int_{\R^n\setminus2B}
    \frac{|x-y|\rho^*(|x-z|)}{|x-z|^{n+1}}
    |b(z)-b_{2B}||f(z)|\,dz   \\
 &=
  \sum_{j=0}^{\infty}
   \int_{2^{j+2}B\setminus2^{j+1}B}
    \frac{|x-y|\rho^*(|x-z|)}{|x-z|^{n+1}}
    |b(z)-b_{2B}||f(z)|\,dz.
\end{align*}
By the doubling condition of $\rho^*$ (see Remark~\ref{rem:rho dec}),
H\"older's inequality
and Lemma~\ref{lem:int f-fB} 
we have
\begin{align*}
 & \int_{2^{j+2}B\setminus2^{j+1}B}
    \frac{|x-y|\rho^*(|x-z|)}{|x-z|^{n+1}}
    |b(z)-b_{2B}||f(z)|\,dz   \\
 &\ls
   \frac{t\rho^*(2^{j+2}t)}{(2^{j+2}t)^{n+1}}
   \int_{2^{j+2}B\setminus2^{j+1}B}
    |b(z)-b_{2B}||f(z)|\,dz   \\
 &\ls
   \frac{\rho^*(2^{j+2}t)}{2^{j+2}}
   \left(\mint_{2^{j+2}B}|b(z)-b_{2B}|^{\eta'}\,dz\right)^{1/\eta'}
   \left(\mint_{2^{j+2}B}|f(z)|^{\eta}\,dz\right)^{1/\eta}\\
 &\le
   \frac{j+2}{2^{j+2}}\|b\|_{\cL_{1,\psi}}
   \left((\rho^*(2^{j+2}t)\psi(2^{j+2}t))^\eta\mint_{2^{j+2}B}|f(z)|^{\eta}\,dz\right)^{1/\eta}.
\end{align*}
Then
\begin{align*}
 |F_3(y)|
 &\ls
  \|b\|_{\cL_{1,\psi}}
  \sum_{j=0}^{\infty}
   \frac{j+2}{2^{j+2}}
   \left((\rho^*(2^{j+2}t)\psi(2^{j+2}t))^\eta\mint_{2^{j+2}B}|f(z)|^{\eta}\,dz\right)^{1/\eta} \\
 &\ls
 \|b\|_{\cL_{1,\psi}}
 \big(M_{(\rho^*\psi)^{\eta}}(|f|^{\eta})(x)\big)^{1/\eta},
\end{align*}
which shows
\begin{equation*}
 \mint_B|F_3(y)|\,dy 
 \ls
 \|b\|_{\cL_{1,\psi}}
 \big(M_{(\rho^*\psi)^{\eta}}(|f|^{\eta})(x)\big)^{1/\eta}.
\end{equation*}
Therefore, we have \eqref{B mean} and the conclusion.
\end{proof}

Next we define the dyadic maximal operator $\Md$.
We denote by $\cQd$ the set of all dyadic cubes, that is,
\begin{equation*}
 \cQd
 =
 \left\{Q_{j,k}=\prod_{i=1}^{n}[2^{-j}k_i,2^{-j}(k_i+1)): j\in\Z,\ k=(k_1,\dots,k_n)\in\Z^n \right\}.
\end{equation*}
Then we define 
\begin{equation*}
  \Md f(x)
  =
  \sup_{R\in\cQd,\,R\ni x}\ \mint_R|f(y)|\,dy,
 \quad x\in\R^n,
\end{equation*}
where the supremum is taken over all $R\in\cQd$ containing $x$.

Next we prove the following proposition.

\begin{prop}\label{prop:sharp LP}
Let $\Phi\in\dtwo$.
If $\Md f\in\LP(\R^n)$,
then
\begin{equation}\label{sharp LP}
 \|\Md f\|_{\LP}
 \le
 C\|\Ms f\|_{\LP}.
\end{equation}
where $C$ is a positive constant which is dependent only on $n$ and $\Phi$.
\end{prop}

The following lemma is well known as the good lambda inequality, 
see \cite[Theorem~3.4.4.]{Grafakos2014GTM250} for example.
\begin{lem}\label{lem:good lambda}
For all $\gamma>0$, all $\lambda>0$, 
and all locally integrable functions $f$ on $\R^n$, 
the following estimate holds.
\begin{equation*}
 |\{x\in\R^n:\Md f(x)>2\lambda, \Ms f(x)\le\gamma\lambda\}|
 \le
 2^n\gamma|\{x\in\R^n:\Md f(x)>\lambda\}|.
\end{equation*}
\end{lem}


\begin{proof}[Proof of Proposition~\ref{prop:sharp LP}]
For a positive real number $N$ we set
\begin{equation*}
 I_N=\int_0^N \Phi'(\lambda)|\{x\in\R^n:\Md f(x)>\lambda\}|\,d\lambda.
\end{equation*}
We note that $I_N\le\int_{\R^n}\Phi(\Md f(x))\,dx<\infty$.
By Lemma~\ref{lem:diff D2} we have
\begin{align*}
 I_N
 &=
 2\int_0^{N/2} \Phi'(2\lambda)|\{x\in\R^n:\Md f(x)>2\lambda\}|\,d\lambda \\
 &\le 
 2C_{\Phi}\int_0^{N/2} \Phi'(\lambda)|\{x\in\R^n:\Md f(x)>2\lambda\}|\,d\lambda.
\end{align*}
Then, using the good lambda inequality, we obtain the following sequence of inequalities:
\begin{align*}
 I_N 
 &\le
 2C_{\Phi} 
 \int_0^{N/2} \Phi'(\lambda) |\{x\in\R^n: \Md f(x) > 2\lambda\ , \Ms f(x) \le \gamma \lambda\ \}|\,d\lambda \\
 &\phantom{***************}
 + 2C_{\Phi} \int_0^{N/2} \Phi'(\lambda) |\{x\in\R^n: \Ms f(x) > \gamma \lambda \}| \,d\lambda \\
 &\le
 2^{n+1} C_{\Phi} \gamma \int_0^{N/2} \Phi'(\lambda) |\{x\in\R^n: \Md f(x) > \lambda\}|\,d\lambda \\
 &\phantom{***************}
 + 2C_{\Phi} \int_0^{N/2} \Phi'(\lambda) |\{x\in\R^n: \Ms f(x) > \gamma \lambda \}| \,d\lambda \\
 &\le 2^{n+1} C_{\Phi} \gamma I_N
 + 2C_{\Phi} \frac{1}{\gamma} 
   \int_0^{N\gamma/2} \Phi'(\lambda/\gamma) |\{x\in\R^n: \Ms f(x) > \lambda \} | \,d\lambda. 
\end{align*}
At this point we let $2^{n+1}C_{\Phi}\gamma = 1/2$. 
Since $I_N$ is finite,
we can substract from both sides of the inequality the quantity $I_N/2$ to obtain
\begin{align*}
 I_N
 &\le
 2^{n+4}C_{\Phi}^2 
 \int_0^{N/(2^{n+3}C_{\Phi})} \Phi'(2^{n+2}C_{\Phi}\lambda) |\{x\in\R^n: \Ms f(x) > \lambda \} | \,d\lambda \\
 &\le
 C_{n,\Phi} 
 \int_0^{\infty} \Phi'(\lambda) |\{x\in\R^n: \Ms f(x) > \lambda \} | \,d\lambda, 
\end{align*}
where $C_{n,\Phi}$ is a constant dependent only on $n$ and $\Phi$,
from which we obtain
\begin{equation*}
 \int_{\R^n}\Phi(\Md f(x))\,dx\le C_{n,\Phi}\int_{\R^n}\Phi(\Ms f(x))\,dx.
\end{equation*}
This shows \eqref{sharp LP}.
\end{proof}

\section{Proof of Theorem~\ref{thm:comm Ir}}\label{sec:proof comm}

We first note that,
for $\theta\in(0,\infty)$,
\begin{equation}\label{theta}
 \||g|^{\theta}\|_{\LP}=\left(\|g\|_{L^{\Phi((\cdot)^{\theta})}}\right)^{\theta}.
\end{equation}

\begin{lem}\label{lem:Lic}
Under the assumption in Theorem~\ref{thm:comm Ir} (i),
if $f\in\Lic(\R^n)$, then $\Ir f\in\LPs(\R^n)$.
\end{lem}

\begin{proof}
If $f\in\Lic(\R^n)$, then $f\in\LP(\R^n)$,
since $\Lic(\R^n)\subset\LP(\R^n)$.
By \eqref{comm Ir A} and Theorem~\ref{thm:Ir}
$\Ir$ is bounded 
from $\LP(\R^n)$ to $L^{\Theta}(\R^n)$.
Then $\Ir f$ is in $L^{\Theta}(\R^n)$.
On the other hand, since $r\mapsto\rho(r)/r^{n-\epsilon}$ is almost decreasing,
if the support of $f$ is in $B(0,R)$, then
\begin{equation*}
 |\Ir f(x)|
 \le
 \|f\|_{L^{\infty}}\int_{B(0,R)}\frac{\rho(|x-y|)}{|x-y|^{n-\epsilon}}\,dy
 \ls
 \|f\|_{L^{\infty}}\int_0^R\frac{\rho(t)}{t^{1-\epsilon}}\,dt<\infty.
\end{equation*}
Then $\Ir f$ is in $L^{\Theta}(\R^n)\cap\Li(\R^n)$.

Next, by \eqref{comm Mr A} and the almost increasingness of $\psi$
we have
\begin{equation*}
 \Theta^{-1}(1/r^n)
 \ls
 \frac{\Psi^{-1}(1/r^n)}{\psi(r)}
 \ls
 \frac{\Psi^{-1}(1/r^n)}{\psi(1)}
 \quad\text{for}\quad r\ge1,
\end{equation*}
and then
\begin{equation*}
 \Theta^{-1}(u)
 \ls
 \Psi^{-1}(u)
 \quad\text{for}\quad u\le1.
\end{equation*}
Hence, we conclude that
\begin{equation*}
 \Psi(t)
 \le
 \begin{cases}
  \Theta(Ct), & t\le1,\\
  \infty, & t>1,
 \end{cases}
\end{equation*}
which shows that $L^{\Theta}(\R^n)\cap\Li(\R^n)\subset\LPs(\R^n)$.
\end{proof}

\begin{proof}[Proof of Theorem~\ref{thm:comm Ir} {\rm (i)}]
We may assume that $\Phi,\Psi\in\dtwo\cap\ntwo$ and $\Theta\in\ntwo$.
We may also assume that $b$ is real valued,
since the commutator $[b,\Ir ]f$ is linear with respect to $b$
and $\|\Re(b)\|_{\cL_{1,\psi}},\|\Im(b)\|_{\cL_{1,\psi}}\le\|b\|_{\cL_{1,\psi}}$.
Let 
\begin{equation*}
 b_k(x)=
 \begin{cases}
  k, & \text{if} \ b(x)>k, \\
  b(x), &  \text{if} \ -k\le b(x)\le k, \\
  -k, &  \text{if} \ b(x)<-k.
 \end{cases}
\end{equation*}
Then $b_k\in L^{\infty}(\R^n)$ and $\|b_k\|_{\cL_{1,\psi}}\le(9/4)\|b\|_{\cL_{1,\psi}}$.
For $f\in\Cic(\R^n)$, $b_kf$ lies in $\Lic(\R^n)$, thus $\Ir(b_kf)$ lies in $\LPs(\R^n)$
by Lemma~\ref{lem:Lic}.
Likewise, $b_k\Ir f$ also lies in $\LPs(\R^n)$.
Since $\Psi\in\ntwo$, $\Md[b,\Ir]f$ is also in $\LPs(\R^n)$.
From this fact and Propositions~\ref{prop:pointwise Ir} and \ref{prop:sharp LP}
it follows that
\begin{align*}
 \|[b_k,\Ir]f\|_{\LPs}
 &\le
 \|\Md([b_k,\Ir]f)\|_{\LPs}
 \ls
 \|\Ms([b_k,\Ir]f)\|_{\LPs} \\
 &\ls
 \|b\|_{\cL_{1,\psi}}
  \bigg(
  \left\|\big(M_{\psi^\eta}(|\Ir f|^{\eta})\big)^{1/\eta}\right\|_{\LPs}
  +\left\|\big(M_{(\rho^*\psi)^{\eta}}(|f|^{\eta})\big)^{1/\eta}\right\|_{\LPs}
 \bigg),
\end{align*}
here, we can choose $\eta\in(1,\infty)$ such that 
$\Phi((\cdot)^{1/\eta})$, $\Psi((\cdot)^{1/\eta})$ and $\Theta((\cdot)^{1/\eta})$ are in $\bntwo$
by Lemma~\ref{lem:eta}.
We show that
\begin{equation*}
  \left\|\big(M_{\psi^\eta}(|\Ir f|^{\eta})\big)^{1/\eta}\right\|_{\LPs}
  +\left\|\big(M_{(\rho^*\psi)^{\eta}}(|f|^{\eta})\big)^{1/\eta}\right\|_{\LPs}
 \ls
 \|f\|_{\LP},
\end{equation*}
where we note that $\psi^{\eta}$ and $(\rho^*\psi)^{\eta}$ are almost increasing.

By Theorems~\ref{thm:Ir} and \ref{thm:Mr} 
we see that 
$\Ir$ is bounded from $\LP(\R^n)$ to $\LT(\R^n)$
and
$M_{\psi^\eta}$ is bounded from $L^{\Theta((\cdot)^{1/\eta})}(\R^n)$ to $L^{\Psi((\cdot)^{1/\eta})}(\R^n)$,
respectively.
Then, using \eqref{theta}, we have
\begin{align*}
 \left\|\big(M_{\psi^\eta}(|\Ir f|^{\eta})\big)^{1/\eta}\right\|_{\LPs}
 &=
 \left(\left\|M_{\psi^\eta}(|\Ir f|^{\eta})\right\|_{L^{\Psi((\cdot)^{1/\eta})}}\right)^{1/\eta} \\
 &\ls
 \left(\left\||\Ir f|^{\eta}\right\|_{L^{\Theta((\cdot)^{1/\eta})}}\right)^{1/\eta} 
 =
 \|\Ir f\|_{\LT} 
 \ls
 \|f\|_{\LP}.
\end{align*}
From \eqref{comm Ir A} and \eqref{comm Mr A} 
it follows that
\begin{equation*}
 (\rho^*(r)\psi(r))^{\eta}\left({\Phi}^{-1}(1/r^n)\right)^{\eta}
 \le
 A^{2\eta}\left(\Psi^{-1}(1/r^n)\right)^{\eta}. 
\end{equation*}
By using Theorem~\ref{thm:Mr},
we have the boundedness of $M_{(\rho^*\psi)^{\eta}}$ 
from $L^{\Phi((\cdot)^{1/\eta})}$ to $L^{\Psi((\cdot)^{1/\eta})}$.
That is,
\begin{align*}
 \left\|\big(M_{(\rho^*\psi)^{\eta}}(|f|^{\eta})\big)^{1/\eta}\right\|_{\LPs}
 &=
 \left(\left\|M_{(\rho^*\psi)^{\eta}}(|f|^{\eta})\right\|_{L^{\Psi((\cdot)^{1/\eta})}}\right)^{1/\eta} \\
 &\ls
 \left(\left\|| f|^{\eta}\right\|_{L^{\Phi((\cdot)^{1/\eta})}}\right)^{1/\eta} 
 =
 \|f\|_{\LP}.
\end{align*}
Therefore, we obtain 
\begin{equation*}
 \|[b_k,\Ir]f\|_{\LPs}
 \ls
 \|b\|_{\cL_{1,\psi}}
 \|f\|_{\LP}
 \quad\text{for all} \ f\in\Cic(\R^n).
\end{equation*}
By the standard argument (see \cite[p.~240]{Grafakos2014GTM250} for example) 
we deduce that,
for some subsequence of integers $k_j$, 
$[b_{k_j} ,\Ir]f\to [b,\Ir]f$ a.e.
Letting $j\to\infty$ and using Fatou's lemma, 
we have
\begin{equation*}
 \|[b,\Ir]f\|_{\LPs}
 \ls
 \|b\|_{\cL_{1,\psi}}
 \|f\|_{\LP}
 \quad\text{for all} \ f\in\Cic(\R^n).
\end{equation*}
Since $\Cic(\R^n)$ is dense in $\LP(\R^n)$ (see Remark~\ref{rem:D2 n2} (ii)), 
it follows that the
commutator admits a bounded extension on $\LP(\R^n)$ that satisfies \eqref{comm Ir}.
\end{proof}

\begin{proof}[Proof of Theorem~\ref{thm:comm Ir} {\rm (ii)}]
We use the method by Janson~\cite{Janson1978}.
Since $|z|^{n-\alpha}$ is infinitely differentiable in an open set,
we may choose $z_0\ne0$ and $\delta>0$ such that 
$|z|^{n-\alpha}$ can be expressed in the neighborhood $|z-z_0|<2\delta$
as an absolutely convergent Fourier series, $|z|^{n-\alpha}=\sum a_je^{iv_j\cdot z}$.
(The exact form of the vectors $v_j$ is irrelevant.)

Set $z_1=z_0/\delta$. If $|z-z_1|<2$, we have the expansion
\begin{equation*}
 |z|^{n-\alpha}=\delta^{-n+\alpha}|\delta z|^{n-\alpha}
 =\delta^{-n+\alpha}\sum a_je^{iv_j\cdot\delta z}.
\end{equation*}
Choose now any ball $B=B(x_0,r)$.
Set $y_0=x_0-rz_1$ and $B'=B(y_0,r)$.
Then, if $x\in B$ and $y\in B'$,
\begin{equation*}
 \left|\frac{x-y}r-z_1\right|\le\left|\frac{x-x_0}r\right|+\left|\frac{y-y_0}r\right|<2.
\end{equation*}
Denote $\sgn(f(x)-f_{B'})$ by $s(x)$. 
Then
\begin{align*}
 &\int_B|b(x)-b_{B'}|\,dx
 =
 \int_B(b(x)-b_{B'})s(x)\,dx
 =
 \frac1{|B'|}\int_B\int_{B'}(b(x)-b(y))s(x)\,dy\,dx \\
 &=
 \frac1{|B'|}\int_{\R^n}\int_{\R^n}
  (b(x)-b(y))
  \frac{r^{n-\alpha}\left|\frac{x-y}{r}\right|^{n-\alpha}}{|x-y|^{n-\alpha}}
  s(x)\chi_{B}(x)\chi_{B'}(y)\,dy\,dx \\
 &=
 \frac{r^{n-\alpha}\delta^{-n+\alpha}}{|B'|}\int_{\R^n}\int_{\R^n}
  \frac{b(x)-b(y)}{|x-y|^{n-\alpha}} \sum a_je^{iv_j\cdot\delta\frac{x-y}r}
  s(x)\chi_{B}(x)\chi_{B'}(y)\,dy\,dx.
\end{align*}
Here, we set $C=\delta^{-n+\alpha}|B(0,1)|^{-1}$ and
\begin{equation*}
 g_j(y)=e^{-iv_j\cdot\delta\frac{y}r}\chi_{B'}(y),
 \quad
 h_j(x)=e^{iv_j\cdot\delta\frac{x}r}s(x)\chi_{B}(x).
\end{equation*}
Then
\begin{align*}
 \int_B|b(x)-b_{B'}|\,dx 
 &=
 Cr^{-\alpha}\sum a_j 
 \int_{\R^n}\int_{\R^n}
  \frac{b(x)-b(y)}{|x-y|^{n-\alpha}} g_j(y)h_j(x) \,dy\,dx \\
 &=
 Cr^{-\alpha}\sum a_j \int_{\R^n} ([b,\Ia]g_j)(x)h_j(x)\,dx \\
 &\le
 Cr^{-\alpha}\sum |a_j| \int_{\R^n} |([b,\Ia]g_j)(x)||h_j(x)|\,dx \\
 &=
 Cr^{-\alpha}\sum |a_j| \int_{B} |([b,\Ia]g_j)(x)|\,dx \\
 &\le
 2Cr^{-\alpha}\sum |a_j| \|\chi_B\|_{\LcPs} \|[b,\Ia]g_j\|_{\LPs} \\
 &\le
 2Cr^{-\alpha}\|[b,\Ia]\|_{\LP\to\LPs} |B|\Psi^{-1}(|B|^{-1})
 \sum |a_j| \|g_j\|_{\LP}.
\end{align*}
Since $\|g_j\|_{\LP}=\|\chi_{B'}\|_{\LP}=1/\Phi^{-1}(|B'|^{-1})\sim1/\Phi^{-1}(r^{-n})$,
we have
\begin{equation*}
 \frac1{\psi(B)}
 \mint_B|b(x)-b_{B'}|\,dx 
 \ls
 \|[b,\Ia]\|_{\LP\to\LPs}
 \frac{\Psi^{-1}(r^{-n})}{r^{\alpha}\psi(B)\Phi^{-1}(r^{-n})}
 \ls
 \|[b,\Ia]\|_{\LP\to\LPs}.
\end{equation*}
That is, $\|b\|_{\cL^{(1,\psi)}}\ls\|[b,\Ia]\|_{\LP\to\LPs}$
and we have the conclusion.
\end{proof}

\section*{Acknowledgement}
The authors would like to thank the referee for her/his careful reading 
and useful comments.
This research was supported by Grant-in-Aid for Scientific Research (B), 
No.~15H03621, Japan Society for the Promotion of Science.



\begin{thebibliography}{99}



\bibitem{Arai-Nakai2017REMC}
R.~Arai and E.~Nakai,
Commutators of Calder\'on-Zygmund and generalized fractional integral operators on generalized Morrey spaces,
Rev. Mat. Complut. 31 (2018), no.~2, 287--331.
https://doi.org/10.1007/s13163-017-0251-4



\bibitem{Chanillo1982}
S.~Chanillo,
A note on commutators,
Indiana Univ. Math. J. 31 (1982), no.~1, 7--16.


\bibitem{Cianchi1999}
A. Cianchi,  
{Strong and weak type inequalities for some classical operators 
in Orlicz spaces}, 
J. London Math. Soc. (2) 60 (1999), no.~1, 187--202.


\bibitem{Deringoz-Guliyev-Nakai-Sawano-Shi-preprint} 
F.~Deringoz, V.S.~Guliyev, E.~Nakai, Y.~Sawano and M.~Shi,
Generalized fractional maximal and integral operators on Orlicz
and generalized Orlicz--Morrey spaces of the third kind,
Positivity, Online First. \\
http://link.springer.com/article/10.1007/s11117-018-0635-9 \\
https://arxiv.org/abs/1812.03649

\bibitem{Edmunds-Gurka-Opic1995}
D. E. Edmunds, P. Gurka and  B. Opic,
{Double exponential integrability of convolution operators in generalized Lorentz-Zygmund spaces}, 
Indiana Univ. Math. J. 44 (1995), no.~1, 19--43.






\bibitem{Fu-Yang-Yuan2014} 
X.~Fu, D.~Yang and W.~Yuan, 
Generalized fractional integrals and their commutators over non-homogeneous metric measure spaces, 
Taiwanese J. Math. 18 (2014), no.~2, 509--557. 




\bibitem{Grafakos2014GTM250} 
L.~Grafakos,
Modern Fourier analysis, Third edition, 
Graduate Texts in Mathematics, 250. Springer, New York, 2014. 
xvi+624 pp. 

\bibitem{Guliyev-Deringoz-Hasanov2017}
V.~S.~Guliyev, F.~Deringoz and S.~G.~Hasanov, 
Riesz potential and its commutators on Orlicz spaces, 
J. Inequal. Appl. 2017, Paper no.~75, 18 pp. 



\bibitem{Hedberg1972}
L.~I.~Hedberg,
{On certain convolution inequalities}, 
Proc. Amer. Math. Soc. 36 (1972), 505--510. 

\bibitem{Janson1978}
S.~Janson, 
Mean oscillation and commutators of singular integral operators,
Ark. Mat.  16  (1978), no. 2, 263--270. 

\bibitem{Kawasumi-Nakai-preprint}
R.~Kawasumi and E.~Nakai,
Pointwise multipliers on weak Orlicz spaces,
preprint.
https://arxiv.org/abs/1811.02858

\bibitem{Kita1996PAMS}
H.~Kita,
{On maximal functions in Orlicz spaces},
Proc. Amer. Math. Soc. 124 (1996), 3019--3025.

\bibitem{Kita1997MathNachr}
H.~Kita,
{On Hardy-Littlewood maximal functions in Orlicz spaces},
Math. Nachr. 183 (1997), 135--155.

\bibitem{Kita2009}
H.~Kita,
Orlicz spaces and their applications (Japanese),
Iwanami Shoten, Publishers. Tokyo, 2009.


\bibitem{Kokilashvili-Krbec1991}
V.~Kokilashvili and M.~Krbec,
{Weighted inequalities in Lorentz and Orlicz spaces},
World Scientific Publishing Co., Inc., River Edge, NJ, 1991. 

\bibitem{Krasnoselsky-Rutitsky1961}
M,~A.~Krasnoselsky and Y.~B.~Rutitsky, 
Convex functions and Orlicz spaces. 
Translated from the first Russian edition by Leo F. Boron. 
P. Noordhoff Ltd., Groningen 1961 .


\bibitem{Maligranda1989}
L. Maligranda,
Orlicz spaces and interpolation,
Seminars in mathematics 5,
Departamento de Matem\'atica, Universidade Estadual de Campinas, Brasil, 1989.


\bibitem{Mizuta-Nakai-Ohno-Shimomura2010JMSJ}
Y.~Mizuta, E.~Nakai, T.~Ohno and T.~Shimomura, 
Boundedness of fractional integral operators on Morrey spaces and Sobolev embeddings 
for generalized Riesz potentials. 
J. Math. Soc. Japan 62 (2010), no.~3, 707--744. 






\bibitem{Nakai2000ISAAC}
E.~Nakai, 
On generalized fractional integrals in the Orlicz spaces. 
Proceedings of the Second ISAAC Congress, Vol. 1 (Fukuoka, 1999), 75--81, 
Int. Soc. Anal. Appl. Comput., 7, Kluwer Acad. Publ., Dordrecht, 2000. 

\bibitem{Nakai2001Taiwan}
E. Nakai, 
{On generalized fractional integrals}, 
Taiwanese J. Math. 5 (2001), 587--602.

\bibitem{Nakai2001SCMJ}
E. Nakai, 
{On generalized fractional integrals in the Orlicz spaces 
on spaces of homogeneous type}, 
Sci. Math. Jpn. 54 (2001), 473--487.

\bibitem{Nakai2002Lund}
E. Nakai,
{On generalized fractional integrals on the weak Orlicz spaces, 
$\BMO_{\phi}$, the Morrey spaces and the Campanato spaces}, 
Function spaces, interpolation theory and related topics (Lund, 2000), 
de Gruyter, Berlin, 2002, 389--401.

\bibitem{Nakai2004KIT}
E. Nakai,
{Generalized fractional integrals on Orlicz-Morrey spaces}, 
Banach and Function Spaces
(Kitakyushu, 2003), Yokohama Publishers, Yokohama, 2004, 323--333.


\bibitem{Nakai2008Studia}
E.~Nakai,
{Orlicz-Morrey spaces and the Hardy-Littlewood maximal function},
Studia Math. 188 (2008), No~3, 193--221.

\bibitem{Nakai2008AMS}
E.~Nakai,
A generalization of Hardy spaces $H^p$ by using atoms, 
Acta Math. Sin. (Engl. Ser.) 24 (2008), no.~8, 1243--1268.


\bibitem{Nakai-Sumitomo2001SCMJ}
E.~Nakai and H.~Sumitomo,
On generalized Riesz potentials and spaces of some smooth functions,
Sci. Math. Jpn. 54 (2001), no.~3, 463--472.



\bibitem{ONeil1965}
R.~O'Neil,
{Fractional integration in Orlicz spaces. I.}, 
Trans. Amer. Math. Soc. 115 (1965), 300--328. 

\bibitem{Orlicz1932}
W.~Orlicz,
{\"Uber eine gewisse Klasse von R\"aumen vom Typus B},
Bull. Acad. Polonaise A (1932), 207--220; 
reprinted in his Collected Papers, PWN, Warszawa 1988, 217--230.

\bibitem{Orlicz1936}
W.~Orlicz,
{\"Uber R\"aume $(L^M)$},
Bull. Acad. Polonaise A (1936), 93--107; 
reprinted in his Collected Papers, PWN, Warszawa 1988, 345--359.




\bibitem{Rao-Ren1991}
M.~M.~Rao and Z.~D.~Ren,
{Theory of Orlicz Spaces},
Marcel Dekker, Inc., New York, Basel and Hong Kong, 1991.


\bibitem{Sawano-Sugano-Tanaka2011} 
Y.~Sawano, S.~Sugano and H.~Tanaka, 
Generalized fractional integral operators and fractional maximal operators 
in the framework of Morrey spaces, 
Trans. Amer. Math. Soc. 363 (2011), no.~12, 6481--6503.





\bibitem{Strichartz1972}
R. S. Strichartz,
{A note on Trudinger's extension of Sobolev's inequalities}, 
Indiana Univ. Math. J. 21 (1972), 841--842.

\bibitem{Torchinsky1976}
A. Torchinsky, 
{Interpolation of operations and Orlicz classes}, 
Studia Math. 59 (1976), no.~2, 177--207. 


\bibitem{Trudinger1967}
N. S. Trudinger, 
{On imbeddings into Orlicz spaces and some applications}, 
J. Math. Mech. 17 (1967), 473--483.

\bibitem{Weiss1956}
G.~Weiss, 
A note on Orlicz spaces,
Portugal. Math. 15 (1956), 35--47.


\end{thebibliography}
\end{document}